\documentclass[11pt,a4paper]{article}
\usepackage[utf8]{inputenc}
\usepackage{amsmath,amssymb,amsthm}
\usepackage{bm}
\usepackage{graphicx}
\usepackage{enumitem}
\usepackage{hyperref}
\usepackage{geometry}
\usepackage{geometry}
\usepackage{comment}
\usepackage{xcolor}
\numberwithin{equation}{section}
\title{Convergence analysis of structure-preserving schemes for the multicomponent compressible Euler flows}
\author{
Jaya Agnihotri$^{1,*}$, Philipp Öffner$^{1,\dagger}$\\[1ex]
$^{1}$Institut für Mathematik, Clausthal University of Technology, Germany
}

\newtheorem{definition}{Definition}[section]
\newtheorem{proposition}{Proposition}[section]
\newtheorem{lemma}{Lemma}[section]
\newtheorem{remark}{Remark}[section]

\newcommand{\con}{\mathbf{U}}
\newcommand{\conflux}{\mathbf{f}}
\newcommand{\numflux}{\mathbf{F}}

\newcommand{\moment}{\mathbf{m}}
\newcommand{\vel}{\mathbf{u}}
\newcommand{\ener}{\mathcal{E}}
\newcommand{\intener}{\mathbf{e}}
\newcommand{\pressure}{p}
\newcommand{\temp}{\mathbb{T}}
\newcommand{\ent}{\eta}
\newcommand{\entf}{\mathcal{F}}
\newcommand{\entvar}{\mathbf{V}}
\newcommand{\ym}{\mathcal{V}}
\newcommand{\entpot}{\psi}

 \newcommand{\NVAR}{\text{n}}

\usepackage{amsthm}     

\newtheorem{theorem}{Theorem}[section]

\newtheorem{assumption}[theorem]{Assumption}

\theoremstyle{plain}

\theoremstyle{remark}

\numberwithin{thm}{section}
\numberwithin{exm}{section}

\newcommand{\relent}{\mathcal{H}}

\begin{document}

\maketitle
\begingroup
\renewcommand{\thefootnote}{}
\footnotetext[0]{
\begin{tabular}{@{}l@{}}
$^{*}$Corresponding author: jaya.agnihotri@tu-clausthal.de \\
$^{\dagger}$philipp.oeffner@tu-clausthal.de
\end{tabular}
}
\endgroup

\begin{abstract}
We present a convergence analysis of a finite volume (FV) scheme for the 
multicomponent compressible Euler system in the framework 
of dissipative weak (DW) solutions. 
DW solutions were introduced as a generalized solution framework in computational fluid dynamics and have recently gained considerable attention. They extend the well-known Lax Equivalence Theorem to nonlinear settings, meaning that if a numerical scheme is both consistent and stable, it will also converge.
The FV scheme under consideration preserves key physical 
properties of the fluid mixture, in particular, positivity of 
partial densities, pressure, and temperature. Using 
uniform stability bounds and consistency estimates, we prove that the numerical solutions converge in the framework of DW solutions of the multicomponent Euler system. Applying the relative entropy method and the weak-strong uniqueness principle, 
we further show that the approximate solutions converge 
strongly to the classical solution as long as it exists. 
Numerical experiments confirm the theoretical results, not only for low-order FV methods but also through extended numerical investigations of a higher-order, structure-preserving discontinuous Galerkin scheme.
\end{abstract}

\section{Introduction}

Hyperbolic systems of conservation laws describe physical processes in which conserved quantities evolve through transport and wave interaction, with information propagating at finite speed. 
From a modelling perspective, these systems are essential for describing physical processes and engineering applications. Among them, one of the most important and most widely analyzed is the Euler system of gas dynamics, which describes the conservation of mass, momentum, and total energy of an inviscid fluid. In many applications, however, the fluid consists of a mixture of several species, which leads naturally to the multicomponent Euler equations with additional conservation laws for each species. 
Already investigating the single-fluid model either numerically or analytically gives rise to serious challenges. By applying convex integration, non-uniqueness of weak-entropy solutions has been demonstrated for the Euler equations \cite{chiodaroli2015global,de2009euler,feireisl2020oscillatory}, and it is expected that similar results will also hold for the multicomponent Euler system. 
The question of identifying an appropriate solution framework for the Euler systems has received considerable attention in computational fluid dynamics. Various approaches have been proposed in the literature, including measure-valued, dissipative, and weak solution frameworks~\cite{diperna1985measure, fjordholm2017construction, feireisl2016dissipative, gwiazda2015weak, eiter2026}. We refer to these works and the references therein for further discussion and related developments. \\
One of such generalized solution concepts that has attracted close attention in recent years is the concept of dissipative weak (DW) solutions  \cite{feireisl2019uniqueness, feireisl2021numerical}. DW solutions are more general than classical weak solutions, as they satisfy an entropy inequality while allowing defect measures that account for oscillations arising from numerical approximations. 
This class of solutions is sufficiently broad to capture all limits of consistent and stable numerical
schemes \cite{feireisl2021numerical, feireisl2020convergence}. Indeed, the existence of DW solutions can be established via the convergence of suitable
schemes \cite{feireisl2021numerical}. Furthermore, they are
known to coincide with classical solutions whenever such solutions exist (weak–strong uniqueness) or
when additional regularity is available. 
Hence, DW solutions can be regarded as a natural extension
of classical solutions. The concept has been successfully applied to the Euler, Navier-Stokes (NS),   magnetohydrodynamics (MHD),  NS-Kortweg equations and viscous multicomponent
and multiphase compressible fluids, cf. \cite{eiter2026, feireisl2021numerical,   sauerborn2026, bumja2021existence, jin2019,  bangwei2022, novotny2020weak} and references therein. Due to its approximate structure, the DW framework is a natural setting for studying the convergence of numerical schemes and can be viewed as a natural framework to extend the classical Lax Equivalence Theorem to nonlinear problems: consistency and stability of a scheme imply convergence. To establish convergence, it is crucial to preserve key physical properties such as the positivity of density and pressure, which makes structure-preserving schemes an essential ingredient in this analysis~\cite{ DumbserLukacovaThomann2025, zbMATH07559967, kuzmin2025consistency}. \\
In this paper, we follow the approach described in \cite{feireisl2020convergence} for the compressible Euler equations and extend them to the multicomponent system. We introduce DW solutions and prove convergence of a structure-preserving FV scheme, demonstrating key structural properties such as positivity of partial densities and pressure, deriving stability estimates, and analyzing consistency. In addition, the weak–strong uniqueness theorem is adapted to the multicomponent setting for completeness. We verify our theoretical considerations by numerical simulations. Although our analysis focuses on FV methods, similar results are expected for high-order, structure-preserving schemes, as indicated in \cite{abgrall2022convergence, lukavcova2023convergence, kuzmin2025consistency}. For this reason, we also investigate a structure-preserving discontinuous Galerkin (DG) method as described in \cite{renac2021entropy} in our numerical study.

The paper is organized as follows. 
Section~\ref{sec:model} introduces the multicomponent 
compressible Euler system under consideration and the DW solution concept.  The  FV scheme considered is 
presented in Section~\ref{sec:scheme} and the structure-preserving properties are proved. In particular, we demonstrate the positivity of partial densities, pressure and uniform stability 
estimates in detail. Consistency with the weak formulation is 
demonstrated in Section~\ref{sec:consistency}. The limit 
passage and identification of the DW solution are 
carried out in Section~\ref{sec:lim_process} along with the weak-strong uniqueness result,
and strong convergence of the numerical solutions.
Numerical experiments are presented in 
Section~\ref{sec:numerics}. Appendices~\ref{app:entropy_dissipation}-\ref{app:weakBV} collect 
the entropy dissipation proof and weak BV estimate for completeness. 

\section{Model problem}\label{sec:model}

We consider the multicomponent compressible Euler system on the spatial domain,
\begin{align}\label{domain}
\Omega = \Big([0,1]\big|_{\{0,1\}}\Big)^N, \qquad N=1,2,3.    
\end{align}
that is, on the flat torus. This periodic setting avoids boundary contributions and is convenient for the analytical arguments used later. 
Such a model appears, for instance, in the simulation of saturated water-vapour flows \cite{ambroso2007coupling}. The governing equations for a multicomponent compressible mixture are given by:
\begin{subequations}\label{eq:multi_component_euler}
\begin{align}
\partial_t \rho_i + \nabla \cdot(\rho_i \vel) &= 0, \quad i=1,\dots,\NVAR.\\
\partial_t \moment + \nabla \cdot( \rho\vel\otimes \vel + \pressure\mathbb{I}) &= 0, \\
\partial_t \ener + \nabla \cdot\big( (\ener+\pressure) \vel \big) &= 0,
\end{align}
\end{subequations}
where $\rho_i$ denotes the partial density of species $i$, $\rho = \sum_{i=1}^{\NVAR} \rho_i$, $\moment = \rho \vel$, and 
$
\ener = \frac{1}{2}\rho |\vel|^2 + \rho \intener
$
denotes the total energy, with $\intener$ the specific internal energy per unit total mass.

The system~\eqref{eq:multi_component_euler} can be written in conservative form as:
\begin{equation}\label{con_law}
    \partial_t \con + \nabla \cdot \conflux(\con) = 0,
\end{equation}
where 
\begin{equation}\label{con_flux}
    \con = \begin{pmatrix} 
        \rho_1 \\ 
        \vdots \\
        \rho_{\NVAR} \\ 
        \moment \\ 
        \ener 
    \end{pmatrix}, \quad
    \conflux(\con) = \begin{pmatrix} 
        \rho_1 \vel \\ 
        \vdots \\
        \rho_{\NVAR} \vel \\ 
        \rho\vel\otimes \vel + \pressure\mathbb{I} \\ 
        (\ener + \pressure)\vel 
    \end{pmatrix}.
\end{equation}

To close the system, we prescribe an equation of state and caloric relation as in~\cite{gouasmi2020formulation}. The pressure is given by
$
    \pressure = \sum_{i=1}^{\NVAR} \rho_i r_i \temp, 
    \; r_i = \frac{R}{m_i},
$
where $m_i$ is the molar mass of species $i$ and $R$ is the universal gas constant. 

The total energy satisfies
\begin{equation}\label{perfect_gas_assump_multi}
    \ener = \sum_{i=1}^{\NVAR} \rho_i \big(e_{0i} + c_{vi} \temp \big) + \frac{1}{2} \rho |\vel|^2,
\end{equation}
where $e_{0i}$ and $c_{vi}$ denote the reference energy and specific heat at constant volume of species $i$ respectively.
For given conservative variables, the temperature $\temp$ is recovered from~\eqref{perfect_gas_assump_multi}, and the pressure is then obtained from the equation of state.

This thermodynamic closure naturally defines the entropy structure of the system. In particular, following~\cite{giovangigli2012multicomponent,gouasmi2020formulation}, the multicomponent Euler system admits a convex entropy pair $(\ent,\entf)$ with
$
\ent = -\rho s, 
\; 
\entf = -\rho s \,\vel,
$
where $s$ denotes the specific entropy of the mixture. 
For an ideal multicomponent gas, the species entropy is given by
$
s_i = c_{vi} \ln \temp - r_i \ln \rho_i,
$
which yields the mixture entropy
\begin{equation}\label{s_def}
s
=
S(\temp,\rho_1,\rho_2, \ldots, \rho_\NVAR)=
\frac{1}{\rho}
\sum_{i=1}^{\NVAR}
\rho_i
\bigl(
c_{vi}\ln \temp
-
r_i \ln \rho_i
\bigr).
\end{equation}
The multicomponent Euler system satisfies the convex entropy inequality
\begin{equation*}
\partial_t \ent(\con) + \nabla \cdot \entf(\con) \le 0,
\label{eq:convex_entropy}
\end{equation*}
that is,
\begin{equation}\label{ent_inq}
\partial_t
\left(
-
\sum_{i=1}^\NVAR
\rho_i
\left(
c_{v_i}\ln \temp
-
r_i \ln \rho_i
\right)
\right)
+
\nabla \cdot
\left[
-
\moment
\left(
\frac{1}{\rho}
\sum_{i=1}^\NVAR
\rho_i
\left(
c_{v_i}\ln \temp
-
r_i \ln \rho_i
\right)
\right)
\right]
\le 0 .
\end{equation}
We replace \eqref{ent_inq} by the more restrictive renormalized entropy inequality
\begin{equation}\label{eq:renorm-entropy}
\partial_t
\left(
- \rho S_\chi(\temp,\rho_1,\rho_2, \ldots, \rho_\NVAR)
\right)
+
\nabla \cdot
\left[
- \moment S_\chi(\temp,\rho_1,\rho_2,\ldots, \rho_\NVAR)
\right]
\le 0 ,
\end{equation}
where $
S_\chi = \chi \circ S
~\text{and}~
\chi:\mathbb{R}\to\mathbb{R}$
is a non-decreasing concave function\footnote{The introduction of the cut-off function $\chi$ in \eqref{eq:renorm-entropy} is inspired by the approach in \cite{chen2001uniqueness}. Inequality~\eqref{eq:renorm-entropy} can be regarded as a renormalized version of~\eqref{ent_inq}.},
$
\chi \le \bar{\chi}.
$
The system admits a symmetrization through the entropy variables $\entvar = \nabla_\con \ent (\con)$, which can be expressed in terms of thermodynamic quantities.
Finally, we further have 
$\intener(\rho_1, \ldots, \rho_\NVAR, \ent)$ denoted as the specific 
internal energy. It is expressed as a function of the conservative variables, 
obtained by inverting the entropy relation 
$\ent = \sum_{i=1}^\NVAR \rho_i(c_{v_i}\ln\temp - r_i\ln\rho_i)$ 
to recover the temperature
\begin{equation}\label{eq:T-from-eta}
\temp = \exp\!\left(
  \frac{1}{\bar{c}_v}
  \left[
    \frac{\ent}{\rho}
    + \frac{1}{\rho}\sum_{i=1}^\NVAR \rho_i r_i \ln\rho_i
  \right]
\right),
\qquad
\bar{c}_v = \frac{\sum_{i=1}^\NVAR \rho_i c_{v_i}}{\rho},
\end{equation}
and then setting 
$\rho\,\intener = \sum_{i=1}^\NVAR \rho_i(e_i^0 + c_{v_i}\temp)$.

\begin{remark}
In place of the formulation in terms of partial densities $(\rho_i)_{i=1}^{\NVAR}$, 
one may equivalently consider the multicomponent Euler system written in terms of 
the total density $\rho$, velocity $\vel$, total energy, and mass fractions 
$\mathbf{Y}=(Y_1,\dots,Y_{\NVAR - 1})$. 
The conservative variables read
\[
\con=\begin{pmatrix}
\rho \mathbf{Y}  \\ \rho \\ \rho \vel \\ \rho \ener
\end{pmatrix},
\qquad 
\sum_{i=1}^{\NVAR} Y_i = 1, \quad 0 \le Y_i \le 1.
\]
The connection with the partial density formulation is 
\[
\rho = \sum_{i=1}^{\NVAR} \rho_i, 
\qquad 
Y_i = \frac{\rho_i}{\rho}, 
\qquad 
\moment = \rho \vel, 
\qquad 
\rho \ener = \frac{|\moment|^2}{2\rho} + \rho \intener.
\]

For ideal gases with constant heats, mixture thermodynamics is defined by
\begin{align}\label{eq:pres_def}
\intener = c_v(\mathbf{Y})\,\temp, 
\qquad 
\pressure = \rho\, r(\mathbf{Y})\, \temp.    
\end{align}
with
\[
c_v(\mathbf{Y}) = \sum_{i=1}^{\NVAR} Y_i c_{vi}, \quad 
c_p(\mathbf{Y}) = \sum_{i=1}^{\NVAR} Y_i c_{pi}, \quad 
r(\mathbf{Y}) = c_p(\mathbf{Y}) - c_v(\mathbf{Y}),
\]
Both formulations are equivalent through the change of variables~(see~\cite{renac2021entropy} for detailed discussion).
\end{remark}

\subsection*{Dissipative weak  solutions}
The DW solution framework for the Euler and Navier–Stokes equations can be seen as a standard tool for analyzing the convergence of numerical schemes, as it has its foundation in the investigation of sequences of approximate solutions. It can also be interpreted as extending the Lax equivalence theorem to nonlinear problems, meaning that consistency and stability imply convergence.
Moreover, due to weak–strong uniqueness, it provides a natural extension of the classical solutions and is consistent with it. The DW framework has already been extended to other models.
In this work, we further extend the DW framework by introducing a definition for the multicomponent Euler equations that naturally generalizes the single-fluid case~\cite{abgrall2022convergence}.

For the definition, we require the following notation. Let $\mathcal{M}^+(\overline{\Omega})$ denote the set of all positive Radon measures on $\overline{\Omega}$, which can be identified with the space of all linear forms on $C_c(\overline{\Omega})$. For the defect measures in the convergence analysis of~\eqref{eq:multi_component_euler}, we need the space $\mathcal{M}^+(\overline{\Omega}; \mathbb{R}^{N\times N}_{\mathrm{sym}})$ of positive semi-definite matrix-valued measures, consisting of all measures $\mu \in \mathcal{M}^+(\overline{\Omega}, \mathbb{R}^{N\times N}_{\mathrm{sym}})$ satisfying
\begin{equation*}
\int_{\overline{\Omega}} \phi(\xi \otimes \xi) : \mathrm{d}\mu \geq 0 \quad \text{for any } \xi \in \mathbb{R}^N, \, \phi \in C_c(\overline{\Omega}), \, \phi \geq 0.
\end{equation*}
We now define the notion of dissipative weak solutions for the multicomponent Euler system.
\begin{definition}[Dissipative weak solution for the multicomponent 
Euler system]\label{def:DW}
Let $\Omega \subset \mathbb{R}^N$, $N = 1,2,3$, be the bounded periodic 
domain \eqref{domain}. We say that 
$(\rho_1, \rho_2, \ldots, \rho_\NVAR, \moment, \ent)$ is a 
\emph{dissipative weak (DW) solution} of the multicomponent compressible 
Euler system \eqref{eq:multi_component_euler}, with initial data 
$(\rho_1^0, \rho_2^0, \ldots, \rho_\NVAR^0, \moment^0, \ent^0)$, 
if the following conditions hold.

\begin{enumerate}

\item \textbf{Weak continuity.} The solution belongs to the following 
function spaces:
\begin{align}\label{eq:weak_spaces}
\rho_i &\in C_{\mathrm{weak}}([0,T];\, L^1(\Omega)),
\quad \rho_i \geq 0,
\quad i = 1, 2, \ldots, \NVAR, \\
\moment &\in C_{\mathrm{weak}}\!\left([0,T];\,
L^{\frac{2\gamma_{\min}}{\gamma_{\min}+1}}
(\Omega;\mathbb{R}^N)\right), \\
\ent &\in L^\infty(0,T;\, L^1(\Omega))
\cap BV_{\mathrm{weak}}([0,T];\, L^1(\Omega)),
\end{align}
where $\ent = \rho s$ denotes the entropy density and 
$\gamma_{\min} = \min\{\gamma_1, \ldots, \gamma_\NVAR\}$.

\item \textbf{Energy inequality.} There exists an energy defect measure 
$\mathfrak{E} \in L^\infty(0,T;\,\mathcal{M}^+(\overline{\Omega}))$ 
such that the energy inequality
\begin{equation}\label{eq:DW-energy}
\int_\Omega
\left[
  \frac{1}{2}\frac{|\moment|^2}{\rho}
  + \rho\,\intener\!\left(\rho_1, \ldots, \rho_\NVAR,\, \ent\right)
\right](\tau,\cdot)\, dx
+
\int_{{\Omega}} d\mathfrak{E}(\tau)
\leq
\int_\Omega
\left[
  \frac{1}{2}\frac{|\moment^0|^2}{\rho^0}
  + \rho^0\intener\!\left(\rho_1^0, \ldots, \rho_\NVAR^0,\, \ent^0\right)
\right] dx
\end{equation}
holds for a.a.\ $0 \leq \tau \leq T$.

\item \textbf{Equations of continuity.} For each 
$i = 1, 2, \ldots, \NVAR$, the integral identity
\begin{equation}\label{eq:DW-continuity}
\left[\int_\Omega \rho_i\,\varphi\, dx\right]_{t=0}^{t=\tau}
=
\int_0^\tau\!\int_\Omega
\left[
  \rho_i\,\partial_t\varphi
  + \rho_i \vel \cdot \nabla_x\varphi
\right] dx\, dt
\end{equation}
holds for any $0 \leq \tau \leq T$ and any test function 
$\varphi \in C^\infty([0,T]\times\overline{\Omega})$.

\item \textbf{Momentum equation.} Let
$\mathfrak{R} \in L^\infty\!\left(0,T;\,
\mathcal{M}^+\!\left(\overline{\Omega};\,
\mathbb{R}^{N\times N}_{\mathrm{sym}}\right)\right)$
be the Reynolds defect measure. The integral identity
\begin{align}\label{eq:DW-momentum}
\left[\int_\Omega \moment\cdot\boldsymbol{\phi}\, dx
\right]_{t=0}^{t=\tau}
&=
\int_0^\tau\!\int_\Omega
\left[
  \moment\cdot\partial_t\boldsymbol{\phi}
  + \mathbf{1}_{\rho>0}
    \frac{\moment\otimes\moment}{\rho}:\nabla_x\boldsymbol{\phi}
  + \mathbf{1}_{\rho>0}\,
    \pressure\!\left(\rho_1,\ldots,\rho_\NVAR,\,\ent\right)
    \mathrm{div}_x\,\boldsymbol{\phi}
\right] dx\, dt \nonumber\\
&\qquad
+ \int_0^\tau\!\int_\Omega
  \nabla_x\boldsymbol{\phi} : d\mathfrak{R}(t)\, dt
\end{align}
holds for any $0 \leq \tau \leq T$ and any test function 
$\boldsymbol{\phi} \in C^\infty([0,T]\times\overline{\Omega};\mathbb{R}^N)$.

\item \textbf{Entropy balance.} Let
$\{\ym_{t,x}\}_{(t,x)\in(0,T)\times\Omega}$
be a parametrized measure with
\begin{align}\label{eq:DW-measure}
\ym
&\in L^\infty\!\left((0,T)\times\Omega;\,
\mathcal{P}(\mathbb{F})\right),
\quad
\mathbb{F}
= \bigl\{
    \tilde{\rho}_i \in \mathbb{R},\;
    \tilde{\boldsymbol{m}} \in \mathbb{R}^N,\;
    \tilde{\ent} \in \mathbb{R}
  \bigr\}, \notag\\
\langle \ym;\,\tilde{\rho}_i\rangle &= \rho_i,
\qquad
\langle \ym;\,\tilde{\moment}\rangle = \moment,
\qquad
\langle \ym;\,\tilde{\ent}\rangle = \ent, \notag\\
\ym_{t,x}&\!\left\{
  \tilde{\rho}_i \geq 0,\;
  \frac{\tilde{\ent}}{\tilde{\rho}} \geq \underline{s}
\right\} = 1
\quad\text{for a.a.\ }(t,x)\in(0,T)\times\Omega,
\end{align}
where $\tilde{\rho} = \sum_{i=1}^\NVAR \tilde{\rho}_i$ and 
$\underline{s} \in \mathbb{R}$ is a lower bound on the specific entropy.
The integral inequality
\begin{equation}\label{eq:DW-entropy}
\begin{aligned}
\left[
  \int_\Omega \ent\,\varphi\, dx
\right]_{\tau_1-}^{\tau_2+}
&\leq
\int_{\tau_1}^{\tau_2}\!\int_\Omega
\left[
  \ent\,\partial_t\varphi
  +
  \left\langle \ym;\;
    \mathbf{1}_{\tilde{\rho}>0}\,
    \frac{\tilde{\ent}}{\tilde{\rho}}\,
    \tilde{\moment}
  \right\rangle
  \cdot \nabla_x\varphi
\right] dx\, dt\\
\ent(0-,\cdot) &= \ent^0
\end{aligned}
\end{equation}
holds for any $0 \leq \tau_1 \leq \tau_2 < T$ and any 
$\varphi \in C_c^\infty((0,T)\times{\Omega})$, $\varphi \geq 0$.

\item \textbf{Defect compatibility conditions.} For some constants 
$0 < \underline{k} \leq \bar{k}$,
\begin{equation}\label{eq:DW-defect}
\underline{k}\,\mathfrak{E}
\leq \mathrm{tr}[\mathfrak{R}]
\leq \bar{k}\,\mathfrak{E}.
\end{equation}

\end{enumerate}
\end{definition}




\section{Structure-preserving FV scheme}\label{sec:scheme}

We introduce the  semi-discrete FV scheme  under consideration for the 
multicomponent Euler system \eqref{eq:multi_component_euler}.
We discretize the domain $\Omega$ into a uniform Cartesian mesh $\mathcal{T}_h$ with mesh size $h > 0$ for simplicity\footnote{The following investigation also applies to FV methods on more general grids, e.g., shape-regular triangular grids.}. Each cell $K \in \mathcal{T}_h$ has measure 
$|K| = h^N$. We denote by $\mathcal{E}_h$ the set of all interfaces 
$S_{KL} = \partial K \cap \partial L$ with $(N-1)$-dimensional 
measure $|S_{KL}| = h^{N-1}$, and by $\mathbf{n}_{KL}$ the unit 
outward normal from $K$ to $L$. The space of piecewise constant 
functions on $\mathcal{T}_h$ is denoted $X(\mathcal{T}_h)$, and 
the numerical solution $\con_h(t) \in X(\mathcal{T}_h)^N$ takes 
the constant value $\con_K(t)$ on each cell $K$. Initial data 
are set by cell averaging $\con_K(0) = (\Pi_h \con^0)_K$, where 
the projection operator $\Pi_h: L^1(\Omega) \to X(\mathcal{T}_h)$ 
is defined by
\begin{align}\label{eq:projection}
(\Pi_h \phi)_K := \frac{1}{|K|}\int_K \phi(x)\,dx.
\end{align}
The semi-discrete FV scheme reads
\begin{equation}\label{eq:semi_scheme}
|K|\frac{d}{dt}\con_K(t) 
+ \sum_{L \in \mathcal{N}(K)} |S_{KL}|\,\numflux_{KL}(t) = 0,
\qquad t > 0,
\end{equation}
where $\mathcal{N}(K)$ denotes the set of cells sharing a face 
with $K$, and $\numflux_{KL} = \numflux_h(\con_K, \con_L)$ is 
the numerical flux across the interface, assumed to be consistent, 
conservative, and locally Lipschitz continuous. The 
discrete divergence operator and cell-face averages appearing in 
the scheme are defined by
\begin{align}\label{eq:tilde-div}
\bigl(\widetilde{\operatorname{div}}_h\,\mathbf{g}_h\bigr)_K
:= \frac{1}{|K|}\sum_{L \in \mathcal{N}(K)}
|S_{KL}|\,\mathbf{g}_{KL}\cdot\mathbf{n}_{KL},
\qquad
\mathbf{g}_{KL} := \frac{\mathbf{g}_K + \mathbf{g}_L}{2}.  \end{align}
We employ the Lax-Friedrichs flux in the direction of the 
interface normal, defined by
\begin{equation}\label{eq:LF_flux_def}
\numflux_{KL}
= \frac{\conflux(\con_K) + \conflux(\con_L)}{2}
\cdot\mathbf{n}_{KL}
- \frac{\lambda_{KL}}{2}\,(\con_L - \con_K),
\end{equation}
which satisfies the consistency, conservation, and Lipschitz 
properties stated above; see~\cite{feireisl2020convergence}. 
The viscosity coefficient is set to
\begin{equation}\label{local_diff}
\lambda_{KL} 
= \max\bigl(|\vel_K| + c_{\mathrm{mix}}(\con_K),\;
|\vel_L| + c_{\mathrm{mix}}(\con_L)\bigr),
\end{equation}
where $\vel_K = \moment_K/\rho_K$ is the cell velocity and 
$c_{\mathrm{mix}}$ is the frozen mixture sound speed, given by
\[
c_{\mathrm{mix}}^2 
= \gamma_{\mathrm{mix}}\,\frac{\pressure}{\rho},
\qquad
\gamma_{\mathrm{mix}} 
= \frac{\displaystyle\sum_{i=1}^{\NVAR} \rho_i c_{pi}}
{\displaystyle\sum_{i=1}^{\NVAR} \rho_i c_{vi}},
\]
with $\rho = \sum_{i=1}^{\NVAR}\rho_i$ the total mixture density and $\gamma_{\mathrm{mix}}$ the effective ratio of specific heats of the mixture. The local viscosity coefficient $\lambda_{KL}$ is chosen to dominate the largest wave speed at the interface. 
In the stability and 
consistency analysis of Section~\ref{sec:consistency}, 
it is convenient to work with the global upper bound
\begin{equation}\label{eq:global-lambda}
\lambda = \max_{K \in \mathcal{T}_h}
\bigl(|\vel_K| + c_{\mathrm{mix}}(\con_K)\bigr),
\end{equation}
which satisfies $\lambda_{KL} \leq \lambda$ for all 
$K,L \in \mathcal{E}_h$ and is used to obtain uniform 
estimates independent of the interface.


The mathematical entropy pair $(\ent, \entf)$ 
and the entropy variables $\entvar = \nabla_{\con}\ent(\con)$ for the 
multicomponent Euler system~\eqref{eq:multi_component_euler} are derived from the Gibbs relation \cite{gouasmi2020formulation}. 
The entropy 
potential is defined as
\[
\psi(\entvar) = \entvar^T\conflux(\con) - \entf(\con).
\]
The scheme \eqref{eq:semi_scheme} is said to be \emph{entropy 
stable} if there 
exists a consistent numerical entropy flux 
$\hat{\entf}_{KL} = \hat{\entf}_h(\con_K,\con_L)$ such that the 
cellwise discrete entropy inequality
\begin{equation}\label{disc_entrorpy_inq}
|K|\frac{d\ent_K}{dt} 
+ \sum_{L \in \mathcal{N}(K)}|S_{KL}|\,\hat{\entf}_{KL} \leq 0
\end{equation}
holds for all $K \in \mathcal{T}_h$ and $t > 0$. The inequality 
\eqref{disc_entrorpy_inq} is the discrete counterpart of the 
entropy inequality \eqref{eq:renorm-entropy}. Following~\cite{tadmor2003entropy}, a general consistent numerical entropy 
flux is given by
\begin{equation}\label{eq:num_ent_flux}
\hat{\entf}_{KL} 
= \overline{(\entvar_h)}_{KL}\,\numflux_{KL} 
- \overline{\bigl(\psi(\entvar_h)\bigr)}_{KL},
\end{equation}
where $\overline{(\cdot)}_{KL}$ denotes the arithmetic average 
between cells $K$ and $L$. In the present work,  we employ the 
Lax-Friedrichs flux \eqref{eq:LF_flux_def}, for which~\eqref{eq:num_ent_flux} reduces to the 
explicit form
\begin{equation}\label{ent_prod2}
\hat{\entf}_{KL} 
= \frac{\entf_K + \entf_L}{2}\cdot\mathbf{n}_{KL}
- \frac{\lambda_{KL}}{2}\,(\ent_L - \ent_K).
\end{equation}
The scheme \eqref{eq:semi_scheme}-\eqref{eq:LF_flux_def} equipped with \eqref{ent_prod2} satisfies the discrete entropy inequality \eqref{disc_entrorpy_inq}, which follows from the convexity of $\ent$ via a Taylor expansion argument. For completeness, the detailed proof is provided in Appendix~\ref{app:entropy_dissipation}. In the following subsection, we establish several additional important properties of the scheme.

Furthermore, as a consequence of integrating \eqref{disc_entrorpy_inq} in time, the solution $\con_h$ of scheme~\eqref{eq:semi_scheme}-\eqref{eq:LF_flux_def} satisfies the weak BV estimate,
\begin{equation}\label{BV_bound}
\int_0^T \sum_{KL \in \mathcal{E}_h}
\lambda_{KL}\,\bigl|\con_h(t)_L - \con_h(t)_K\bigr|\,h^N\,dt
\;\longrightarrow\; 0
\qquad \text{as } h \to 0^+
\end{equation}
where the proof can be found in Appendix~\ref{app:weakBV}.  The 
estimate \eqref{BV_bound} will be the key compactness ingredient 
for the limit passage carried out in 
Section~\ref{sec:lim_process}.
\subsection{ Properties of the FV 
Scheme}\label{sec:structure}

In this section, we establish further key structure-preserving properties of the FV scheme~\eqref{eq:semi_scheme}-\eqref{eq:LF_flux_def} for completeness, namely the conditional positivity of the partial densities and pressure (see Theorem~\ref{thm:density_positivity}, Lemma~\ref{pos_pressure}), and uniform a priori bounds on all primitive variables. 
We begin by proving the positivity of the partial densities. The argument follows the approach developed for the single-fluid case in \cite{feireisl2020convergence}, suitably adapted to the multicomponent setting.

\begin{theorem}\label{thm:density_positivity}
Let $\rho_{i,K}(0) > 0$ for all $K \in \mathcal{T}_h$, $i = 1, 2, \ldots, \NVAR$, 
and let
$
    \vel_h = \moment_h/\rho_h \in L^2\bigl(0,T;\, L^\infty(\Omega)\bigr).
$
Then the numerical 
solution generated by the scheme 
\eqref{eq:semi_scheme}-\eqref{eq:LF_flux_def} satisfies
\[
\rho_{i,K}(t) \geq \bar{\rho}_h> 0
\qquad \text{for all } K \in \mathcal{T}_h,\ t \in [0,T],\ i = 1, 2, \ldots, \NVAR.
\]

\end{theorem}

\begin{proof}
Fix a component $i \in \{i = 1, 2, \ldots, \NVAR\}$ and let $K^*$ be the cell achieving 
the minimum of $\rho_{i,K}(t)$ over all cells at time $t$, so that
\begin{equation}\label{eq:pos_diff}
(\rho_{i,L} - \rho_{i,K^*}) \geq 0 
\qquad \text{for all } L \in \mathcal{N}(K^*).
\end{equation}
The $\rho_i$-component of scheme 
\eqref{eq:semi_scheme}-\eqref{eq:LF_flux_def} at cell $K^*$ reads
\begin{equation}\label{evolutionary_eq}
\frac{d}{dt}\rho_{i,K^*} = -\frac{1}{|K^*|}\sum_{L \in \mathcal{N}(K^*)}
\frac{|S_{K^*L}|}{2}
\Bigl[(\rho_{i,K^*}\vel_{K^*} + \rho_{i,L}\vel_L)\cdot\mathbf{n}_{K^*L} 
- \lambda(\rho_{i,L} - \rho_{i,K^*})\Bigr].
\end{equation}
Decomposing $\rho_{i,L} = \rho_{i,K^*} + (\rho_{i,L} - \rho_{i,K^*})$ in the convective term and substituting into \eqref{evolutionary_eq} 
yields
\begin{equation}\label{eq:AB_split}
\frac{d}{dt}\rho_{i,K^*}
= -{\rho_{i,K^*}
\bigl(\widetilde{\operatorname{div}}_h\,\vel_h\bigr)_{K^*}}
{-\frac{1}{|K^*|}\sum_{L \in \mathcal{N}(K^*)}
\frac{|S_{K^*L}|}{2}
(\rho_{i,L} - \rho_{i,K^*})
\bigl[\vel_L\cdot\mathbf{n}_{K^*L} - \lambda\bigr]},
\end{equation}

By \eqref{eq:global-lambda} and \eqref{eq:pos_diff}, each factor 
$(\rho_{i,L} - \rho_{i,K^*}) \geq 0$ and 
$[\vel_L\cdot\mathbf{n}_{KL} - \lambda] \leq 0$ therefore,
discarding nonnegative contribution from \eqref{eq:AB_split} gives
\[
\frac{d}{dt}\rho_{i,K^*}(t) \geq -C(t)\,\rho_{i,K^*}(t),
\qquad C(t) := \bigl(\widetilde{\operatorname{div}}_h\,\vel_h\bigr)_{K^*}.
\]
Since $C \in L^2\bigl(0,T; L^\infty(\Omega)\bigr)$ as a consequence from  $\vel_h \in L^2\bigl(0,T;\, L^\infty(\Omega)\bigr)$, an application of Grönwall's inequality yields
\[
\rho_{i,K^*}(t) \geq \rho_{i,K^*}(0)
\exp\!\left(-\int_0^t C(\tau)\,d\tau\right) > 0,
\]
Since $K^*$ was the global minimizer, 
$\rho_{i,L}(t) \geq \rho_{i,K^*}(t) > 0$
for all $L \in \mathcal{T}_h$. Applying the argument independently 
to each $i = 1, 2, \ldots, \NVAR$ and summing gives 
$\rho_K = \sum_{i=1}^{\NVAR} \rho_{i,K} > 0$.
\end{proof}

We next show that the discrete pressure and temperature remain 
positive along the evolution. The argument relies on the discrete 
version of renormalized entropy inequality~\eqref{eq:renorm-entropy} together with the positivity of 
the discrete densities established in 
Theorem~\ref{thm:density_positivity}. We first introduce the 
relevant discrete thermodynamic quantities. The cell temperature is recovered from the conservative variables by inverting the energy equation \eqref{perfect_gas_assump_multi}: 
\begin{equation}\label{eq:Tk_def}
\temp_K(t) = \frac{\ener_K(t) - \Big(\dfrac{1}{2}
\dfrac{|\moment_K(t)|^2}{\rho_K(t)} + \sum_{i=1}^{n} \rho_{i,K}(t)\,e_{0i}\Big)}
{\sum_{i=1}^{\NVAR} \rho_{i,K}(t)\,c_{vi}},
\end{equation}
The specific mixture entropy is given by,
\begin{equation}\label{eq:sk_def}
s_K(t) = \frac{1}{\rho_K(t)}\sum_{i=1}^{\NVAR} \rho_{i,K}(t)
\Bigl(c_{vi}\ln\temp_K(t) - r_i\ln\rho_{i,K}(t)\Bigr),
\end{equation}
and the pressure is defined by
\begin{equation}\label{eq:pk_def}
\pressure_K(t) = \sum_{i=1}^{\NVAR} \rho_{i,K}(t)r_i\,\temp_K(t).
\end{equation}

\begin{lemma}\label{pos_pressure}
Under the hypotheses of Theorem~\ref{thm:density_positivity}, 
assume additionally that the initial entropy satisfies
\begin{equation}\label{eq:entropy_bound}
s_K(0) \geq Z_* \qquad \text{for all } K \in \mathcal{T}_h,
\end{equation}
for some $Z_* \in \mathbb{R}$. Then
\[
\temp_K(t) > 0 \qquad \text{and} \qquad \pressure_K(t) > 0
\qquad \text{for all } K \in \mathcal{T}_h,\ t \in [0,T].
\]
\end{lemma}

\begin{proof}
Rewrite the renormalized entropy inequality~\eqref{eq:renorm-entropy} in discrete form:
\begin{equation}\label{eq:disc_ren_entropy}
\sum_{K \in \mathcal{T}_h} |K|
\Bigl(-\rho_K(\tau)\,\chi\bigl(s_K(\tau)\bigr)\Bigr)
\leq
\sum_{K \in \mathcal{T}_h} |K|
\Bigl(-\rho_K(0)\,\chi\bigl(s_K(0)\bigr)\Bigr),
\end{equation}
for any non-decreasing concave $\chi \leq \bar{\chi}$. Fix $Z_*$ as 
in~\eqref{eq:entropy_bound} and following~\cite{bvrezina2018measure} we now choose a approximation of the cutoff function 
$\chi : \mathbb{R} \to \mathbb{R}$ satisfying
\begin{equation}\label{eq:chi_cutoff}
\chi'(z) \geq 0, \qquad
\chi(z) = \begin{cases} < 0, & z < Z_*, \\ 0, & z \geq Z_*. \end{cases}
\end{equation}
Since \eqref{eq:entropy_bound} implies $s_K(0) \geq Z_*$ for all 
$K \in \mathcal{T}_h$, we have $\chi(s_K(0)) = 0$, therefore:
\begin{equation}\label{eq:sum_eta0_zero}
\sum_{K \in \mathcal{T}_h} |K|
\Bigl(-\rho_K(0)\,\chi\bigl(s_K(0)\bigr)\Bigr) = 0.
\end{equation}
Substituting \eqref{eq:sum_eta0_zero} into \eqref{eq:disc_ren_entropy}:
\begin{equation}\label{eq:sum_eta_t_le0}
\sum_{K \in \mathcal{T}_h} |K|
\Bigl(-\rho_K(t)\,\chi\bigl(s_K(t)\bigr)\Bigr) \leq 0.
\end{equation}
By \eqref{eq:chi_cutoff}, $\chi(s_K(t)) \leq 0$, while 
by~\ref{thm:density_positivity}, $\rho_K(t) > 0$. Hence each term 
$-\rho_K(t)\,\chi(s_K(t)) \geq 0$. Since the sum of 
non-negative terms are bounded above by zero 
in~\eqref{eq:sum_eta_t_le0}, each term must vanish individually, 
giving $\chi(s_K(t)) = 0$ for all $K \in \mathcal{T}_h$ and 
$t \in [0,T]$. By~\eqref{eq:chi_cutoff}, this implies
\begin{equation}\label{eq:s_lower_bound_all_t}
s_K(t) \geq Z_* \qquad 
\text{for all } K \in \mathcal{T}_h,\ t \in [0,T].
\end{equation}

Rewriting \eqref{eq:sk_def} as
\begin{equation}\label{eq:s_expand}
s_K(t) = \bar{c}_{v,K}(t)\,\ln\temp_K(t)
- \frac{1}{\rho_K(t)}\sum_{i=1}^{\NVAR} 
\rho_{i,K}(t)\,r_i\ln\rho_{i,K}(t),
\end{equation}
where
\[
\bar{c}_{v,K}(t) = 
\frac{\sum_{i=1}^{\NVAR} \rho_{i,K}(t)\,c_{vi}}{\rho_K(t)} > 0,
\]
and combining with \eqref{eq:s_lower_bound_all_t}, we obtain
\[
\bar{c}_{v,K}(t)\,\ln\temp_K(t) \geq Z_* 
+ \frac{1}{\rho_K(t)}\sum_{i=1}^{\NVAR} 
\rho_{i,K}(t)\,r_i\ln\rho_{i,K}(t).
\]
Dividing by $\bar{c}_{v,K}(t) > 0$ and exponentiating yields
\begin{equation}\label{eq:T_lower_bound}
\temp_K(t) \geq \exp\!\left(\frac{1}{\bar{c}_{v,K}(t)}
\left[Z_* + \frac{1}{\rho_K(t)}\sum_{i=1}^{\NVAR} 
\rho_{i,K}(t)\,r_i\ln\rho_{i,K}(t)\right]\right) > 0.
\end{equation}

From \eqref{eq:pk_def} and $r_i > 0$ we have 
$\sum_{i=1}^{\NVAR} \rho_{i,K}(t)r_i> 0$, which together 
with \eqref{eq:T_lower_bound} gives $\pressure_K(t) > 0$ for 
all $K \in \mathcal{T}_h$ and $t \in [0,T]$.
\end{proof}
Theorem~\ref{thm:density_positivity} and 
Lemma~\ref{pos_pressure} establish positivity of density and pressure for each mesh size $h$. 
For convergence analysis, these bounds must be 
uniform in $h$, which requires the following assumption:
\begin{assumption}\label{as_1}
We assume that there exist two positive constants 
$\underline{\varrho}$ and $\bar{\mathcal{E}}$ such that
\begin{equation}\label{assump}
\quad
0< \bar{\varrho} \leq \varrho_{i,h}(t)
\qquad
\mathcal{E}_h(t) \le \bar{\mathcal{E}}
\quad \text{uniformly for } h \to 0 .
\end{equation}
\end{assumption}
The physical meaning of this assumption is that no vacuum appears. The
second assumption in \eqref{assump} implies as well then that the speed $|\mathbf{u}_h|$ is bounded since
\[
|\mathbf{u}_h|^2 \le \frac{2\mathcal{E}_h}{\rho_h}
\le \frac{2 \bar{\mathcal{E}}}{\bar{\rho}}
< C.
\]
This assumption implies that the density is
also bounded from above and the energy is bounded from below.
Consequently, the pressure and temperature are bounded from above
and below as well, cf. \cite{yuan2023, lukavcova2023convergence}. Note that Assumption \ref{as_1} has been made up to this point in every convergence analysis for the single fluid Euler case, see \cite{maria2025} for a summary. 


Next, we derive a priori bounds on the discrete solution, 
following~\cite{feireisl2020convergence}. By 
Assumption~\ref{as_1}, these bounds will be uniform in $h$. Since the 
$\rho_i$-component of scheme 
\eqref{eq:semi_scheme}-\eqref{eq:LF_flux_def} is in 
conservative form, summing over all cells gives exact mass 
conservation. Together with $\rho_{i,h} > 0$ from 
Theorem~\ref{thm:density_positivity}, this yields
\[
\|\rho_{i,h}\|_{L^\infty(0,T;\,L^1(\Omega))} 
= \|\rho_{i,h}(0)\|_{L^1(\Omega)} \leq C, 
\qquad i = 1, 2, \ldots, \NVAR.
\]
similarly, summing the $\ener$-component of 
\eqref{eq:semi_scheme}-\eqref{eq:LF_flux_def} over all cells 
and using the conservative structure under periodic boundary 
conditions give:
\begin{equation}\label{ener_bound}
\|\ener_h\|_{L^\infty(0,T;\,L^1(\Omega))} 
= \|\ener_h(0)\|_{L^1(\Omega)} \leq C.
\end{equation}
By recalling the pressure law~\eqref{eq:pk_def} and noting that $\sum_i \rho_i c_{vi}\temp \leq \ener$ with $r_i = (\gamma_i - 1)c_{vi}$, we obtain
\[
\pressure = \sum_{i=1}^{\NVAR} (\gamma_i - 1)\rho_i c_{vi}\temp 
\leq (\gamma_{\max} - 1)\sum_{i=1}^{\NVAR} \rho_i c_{vi}\temp 
\leq (\gamma_{\max} - 1)\ener_K,
\]
where $\gamma_{\max} = \max\{\gamma_1, \gamma_2,\ldots, \gamma_\NVAR\}$. Integrating over 
$\Omega$ and applying \eqref{ener_bound} yields:
\begin{equation}\label{pressure_bound}
\|\pressure_h\|_{L^\infty(0,T;\,L^1(\Omega))} 
\leq (\gamma_{\max} - 1)\|\ener_h\|_{L^\infty(0,T;\,L^1(\Omega))} 
\leq C.
\end{equation}

\begin{lemma}\label{lem:mix_rho_bound}
Let $(\rho_{1,h}, \rho_{2,h},\ldots, \rho_{\NVAR,h}, \moment_h, \ener_h)$ be the 
discrete solution generated by scheme~\eqref{eq:semi_scheme}-\eqref{eq:LF_flux_def}. 
Under the hypotheses of Lemma~\ref{pos_pressure}, with 
$\gamma_{\min} = \min\{\gamma_1, \gamma_2, \ldots, \gamma_\NVAR\} > 1$, the following 
bounds hold uniformly in $h$ and $t \in [0,T]$:
\begin{align}
\rho_h &\in L^\infty(0,T;\,L^{\gamma_{\min}}(\Omega)),
\label{rho_bound}\\
\moment_h &\in L^\infty(0,T;\,L^k(\Omega;\mathbb{R}^N)),
\qquad k = \frac{2\gamma_{\min}}{\gamma_{\min}+1} > 1.
\label{moment_bound}
\end{align}
\end{lemma}

\begin{proof}
Fix a cell $K$ and time $t$ and omit the indices $(K,t)$ for 
readability. 

Multiplying \eqref{eq:sk_def} by $\rho$ and substituting 
$\ln\temp = \ln\pressure - \ln\rho - \ln(\sum_i Y_i r_i)$
from the pressure law, together with 
$\ln\rho_i = \ln\rho + \ln Y_i$ and the bound 
$y\ln y \geq -1/e$ for $y \in (0,1]$, gives
\begin{equation}\label{eq:key_collect}
\rho s \leq \rho\,c_v(\mathbf{Y})\ln\pressure 
- \rho\,c_p(\mathbf{Y})\ln\rho + C\rho,
\end{equation}
where $c_v(\mathbf{Y}) = \sum_i Y_i c_{vi}$ and 
$c_p(\mathbf{Y}) = \sum_i Y_i c_{pi}$, and $C$ absorbs the bounded remainder terms 
 $-c_v(\mathbf{Y})\ln r(\mathbf{Y}) - \sum_i Y_i r_i \ln Y_i$. 
Since $\rho Z_* \leq \rho s$ from \eqref{eq:s_lower_bound_all_t}, 
combining with \eqref{eq:key_collect} and 
dividing by $\rho\,c_v(\mathbf{Y}) > 0$, then exponentiating gives:
\[
\pressure \geq c\,\rho^{\gamma(\mathbf{Y})}, 
\qquad \gamma(\mathbf{Y}) := \frac{c_p(\mathbf{Y})}{c_v(\mathbf{Y})} 
= \sum_{i=1}^{\NVAR} w_i\gamma_i \geq \gamma_{\min},
\]
where $w_i = Y_i c_{vi}/\sum_j Y_j c_{vj} \geq 0$ with 
$\sum_{i=1}^{\NVAR} w_i = 1$, and $c = \exp\bigl((Z_*-C)/c_v(\mathbf{Y})\bigr) > 0$. Hence for all $\rho > 0$,
\[
\rho^{\gamma_{\min}} \leq 1 + C\,\pressure.
\]
Integrating over $\Omega$ and applying \eqref{pressure_bound} 
gives \eqref{rho_bound}.

\noindent\textit{Momentum bound.}
By H\"{o}lder's inequality with exponents 
$p = \frac{2}{2-k}$, $q = \frac{2}{k}$,
\[
\|\moment_h\|_{L^k(\Omega)}
\leq
\|\rho_h\|_{L^{\frac{k}{2-k}}(\Omega)}^{1/2}
\|\sqrt{\rho_h}\,\vel_h\|_{L^2(\Omega)}.
\]
Choosing $k = \frac{2\gamma_{\min}}{\gamma_{\min}+1}$ so that 
$\frac{k}{2-k} = \gamma_{\min}$, and using \eqref{rho_bound} 
together with 
$\sqrt{\rho_h}\,\vel_h \in L^\infty(0,T;\,L^2(\Omega))$ 
from the energy bound \eqref{ener_bound}, yields 
\eqref{moment_bound}. Note $k > 1$ iff $\gamma_{\min} > 1$.
\end{proof}

Together with the weak BV estimate~\eqref{BV_bound}, the bounds \eqref{rho_bound}-\eqref{moment_bound} enable us to establish convergence later on.

\section{Consistency}\label{sec:consistency}

We prove that the semi-discrete FV scheme
\eqref{eq:semi_scheme}--\eqref{eq:LF_flux_def} yields a consistent
approximation of the multicomponent Euler system
\eqref{eq:multi_component_euler}, following the consistency framework
of~\cite{abgrall2022convergence, feireisl2020convergence, kuzmin2025consistency}. 
For the numerical solution $\con_h = (\rho_{1,h},\ldots,\rho_{\NVAR,h},\moment_h,\ener_h)$,
we show that the consistency residual $\mathcal{R}^h$ defined by
\begin{equation}\label{eq:consist_goal}
\left[\int_\Omega \con_h\cdot\varphi\,dx\right]_{t=0}^{t=\tau}
=\int_0^\tau\!\int_\Omega
\bigl[\partial_t\varphi\cdot\con_h
+\mathbf{f}(\con_h):\nabla_x\varphi\bigr]dx\,dt
+\int_0^\tau \mathcal{R}^h(t,\varphi)\,dt
\end{equation}
holds for all $\varphi\in C^2([0,T]\times\overline{\Omega})$ where the error $\mathcal{R}^h\to 0$ as $h\to 0$. 
Note that the consistency of the total energy (see Definition~\ref{def:DW}, energy inequality) follows directly from the global conservation property of the scheme. Therefore, in the subsequent analysis, we focus only on the species densities $\rho_i$ and the momentum $\moment$. However, establishing the consistency estimates requires the weak BV estimate (Proposition~\ref{prop:weak_BV}), which is based on entropy dissipation. Therefore, we derive the consistency formulation for $\con_h = (\rho_{1,h}, \ldots, \rho_{\NVAR,h}, \moment_h, \eta_h).$

The analysis is independent of the specific choice of numerical flux, such as local or global Lax-Friedrichs. 
\subsection{Consistency error decomposition}\label{subsec:error_decomp}

We derive the error $\mathcal{R}^h$ explicitly, following the structure of~\cite{abgrall2022convergence,kuzmin2025consistency}. For any
$\varphi\in C^2([0,T]\times\overline{\Omega})$, we begin with the identity
\begin{equation}\label{eq:start}
\left[\int_\Omega\con_h\,\varphi\,dx\right]_{t=0}^{t=\tau}
=\int_0^\tau\!\int_\Omega\frac{d}{dt}(\con_h\,\varphi)\,dx\,dt
=\int_0^\tau\!\int_\Omega\bigl[\con_h\,\partial_t\varphi
+\varphi\,\partial_t\con_h\bigr]dx\,dt.
\end{equation}
To evaluate $\int_\Omega\varphi\,\partial_t\con_h\,dx$, we introduce the cell-average projection~\eqref{eq:projection} and decompose:
\begin{equation}\label{eq:decomp_phi}
\int_\Omega\varphi\,\partial_t\con_h\,dx
=\int_\Omega\Pi_h\varphi\,\partial_t\con_h\,dx
+\underbrace{\int_\Omega(\varphi-\Pi_h\varphi)\,\partial_t\con_h\,dx}_{=:\,\mathrm{II}}.
\end{equation}
Substituting the semi-discrete scheme~\eqref{eq:semi_scheme} into the
first integral of \eqref{eq:decomp_phi}:
\begin{equation}\label{eq:sub_scheme}
\int_\Omega\Pi_h\varphi\,\partial_t\con_h\,dx
=\sum_{K\in\mathcal{T}_h}|K|\,(\Pi_h\varphi)_K\,\partial_t\con_K
=-\sum_{K\in\mathcal{T}_h}(\Pi_h\varphi)_K
\sum_{L\in\mathcal{N}(K)}|S_{KL}|\,\numflux_{KL}.
\end{equation}
Applying summation-by-parts under periodic boundary conditions
(using $\numflux_{LK}=-\numflux_{KL}$):
\begin{equation}\label{eq:SBP}
-\sum_{K}(\Pi_h\varphi)_K\sum_{L\in\mathcal{N}(K)}|S_{KL}|\numflux_{KL}
=\sum_{KL\in\mathcal{E}_h}|S_{KL}|\,\numflux_{KL}\cdot
\bigl[(\Pi_h\varphi)_L-(\Pi_h\varphi)_K\bigr].
\end{equation}
Substituting \eqref{eq:LF_flux_def} into \eqref{eq:SBP}, the diffusion contribution defines term~$\mathrm{I}$:
\begin{equation}\label{eq:termI_def}
\mathrm{I}:=-\sum_{KL\in\mathcal{E}_h}|S_{KL}|\,
\frac{\lambda_{KL}}{2}\,(\con_L-\con_K)
\cdot\bigl[(\Pi_h\varphi)_L-(\Pi_h\varphi)_K\bigr].
\end{equation}
For the central flux, a Taylor expansion of $(\Pi_h\varphi)_L-(\Pi_h\varphi)_K$
recovers $-\int_\Omega \conflux(\con_h):\nabla_x(\Pi_h\varphi)\,dx$ plus a remainder
absorbed into $\mathrm{III}$.

Writing $\nabla_x\varphi = \nabla_x(\Pi_h\varphi)
+ \nabla_x(\varphi-\Pi_h\varphi)$, we define
\begin{equation}\label{eq:termIII_def}
\mathrm{III}:=\int_\Omega(\Pi_h\varphi-\varphi):\conflux(\con_h)\,dx.
\end{equation}
Collecting all contributions into \eqref{eq:start}, we obtain
\eqref{eq:consist_goal} with the consistency error:
\begin{align}\label{eq:error_decomp}
\mathcal{R}^h(t,\varphi)
&=\underbrace{-\sum_{KL\in\mathcal{E}_h}|S_{KL}|\,
  \frac{\lambda_{KL}}{2}(\con_L-\con_K)
  \cdot\bigl[(\Pi_h\varphi)_L-(\Pi_h\varphi)_K\bigr]}_{\displaystyle\mathrm{I}}\nonumber\\
&\qquad\qquad+\underbrace{\Bigl[\int_\Omega\con_h(\varphi-\Pi_h\varphi)\,dx
  \Bigr]_{t=0}^{t=\tau}}_{\displaystyle\mathrm{II}}
+\underbrace{\int_\Omega(\Pi_h\varphi-\varphi):\conflux(\con_h)\,dx}_{\displaystyle\mathrm{III}}.
\end{align}
The preceding analysis shows that Terms~$\mathrm{(II)}$ and $\mathrm{(III)}$ tend to zero as $h\to 0$ for bounded sequences $\con_h$. The remaining challenge is to control Term~$\mathrm{I}$, which we address in what follows by demonstrating its convergence to zero. 

The mesh is a uniform Cartesian, neighbouring cell centres satisfy
$|x_K - x_L| \leq Ch$, so
\begin{equation}\label{eq:proj_diff_bound}
\bigl|(\Pi_h\varphi)_L - (\Pi_h\varphi)_K\bigr|
\leq Ch\,\|\nabla\varphi\|_{C([0,T]\times\Omega)},
\end{equation}
and therefore
\begin{equation}\label{eq:termI_bound}
\int_0^T|\mathrm{I}(t)|\,dt
\leq
C\,\|\nabla\varphi\|_{C([0,T]\times\Omega)}
\int_0^T\!\sum_{KL\in\mathcal{E}_h}
\lambda_{KL}\,|\con_L-\con_K|\,h^N\,dt
\;\longrightarrow\; 0
\qquad\text{as }h\to 0^+,
\end{equation}
where the convergence to zero follows from the weak BV
estimate~\eqref{BV_bound}. Hence term~$(\mathrm{I})$
vanishes in the limit and the Lax-Friedrichs diffusion
does not contribute to the limiting integral identities
of Definition~\ref{def:DW}.

\subsection{Entropy inequality consistency}\label{subsec:ent_consist}

For the entropy component $\ent_h$, the scheme satisfies a discrete entropy inequality, tested only against $\varphi\geq 0$:
\begin{equation}\label{eq:disc_ent}
|K|\frac{d\ent_K}{dt}
+\sum_{L\in\mathcal{N}(K)}|S_{KL}|\,\hat{\entf}_{KL}\leq 0.
\end{equation}
Applying the decomposition \eqref{eq:error_decomp} to \eqref{eq:disc_ent}
with $\varphi\geq 0$ yields an error
$\mathcal{R}^h_\ent=\mathrm{(I)}_\ent+\mathrm{(II)}_\ent+\mathrm{(III)}_\ent
+\mathrm{(IV)}_\ent$,
where terms $(\mathrm{I})_\ent$--$(\mathrm{III})_\ent$ are as before.
The additional term $\mathrm{(IV)}_\ent$ arises from replacing the
numerical entropy flux $\hat{\entf}_{KL}$ by the physical flux
$\entf(\con_K)=-\ent_K\vel_K$:
\begin{align*}
\mathrm{IV}_\ent 
&= \int_0^\tau \sum_{KL\in\mathcal{E}_h} |S_{KL}|
   \bigl[\hat{\entf}_{KL} - \entf(\con_K)\cdot\mathbf{n}_{KL}\bigr]
   (\Pi_h\varphi)_K \, dt
\end{align*}
Next, we suppose that 
the numerical entropy
flux is Lipschitz, meaning that 
\begin{equation}\label{eq:lip_ent_flux}
\bigl|\hat{\entf}_{KL}(\con_K,\con_L)-\entf(\con_K)\bigr|
\leq C_L\,|\con_K-\con_L|
\qquad\forall\,K,L\in\mathcal{E}_h,
\end{equation}
holds. Therefore, we
obtain, 
\begin{equation}\label{eq:IV_bound}
|\mathrm{IV}_\ent|
\leq C_L\,\|\varphi\|_{C([0,T]\times\Omega)}
\int_0^T\sum_{KL\in\mathcal{E}_h}
\lambda_{KL}\,|\con_L-\con_K|\,h^N\,dt
\;\longrightarrow\; 0
\end{equation}
together with the BV estimation ~\eqref{BV_bound}. 
All four terms 
($\mathrm{I})_\ent$-$(\mathrm{IV})_\ent$ vanish as $h\to 0$, yielding the entropy consistency. 
To summarize, we can state the following:
\begin{theorem}[Consistency]\label{thm:consistency}
Let $(\rho_{1,h},\ldots,\rho_{\NVAR,h},\moment_h,\ent_h)_{h>0}$ 
be a family of discrete solutions generated by the scheme 
\eqref{eq:semi_scheme}-\eqref{eq:LF_flux_def}, satisfying 
Assumption~\ref{as_1}, the weak BV estimate~\eqref{BV_bound}, 
and Lipschitz continuity of the numerical entropy flux.

Then, for every $\tau \in (0,T)$, the following holds:
\begin{enumerate}

\item \emph{(Continuity, each species)} For each $i=1,\ldots,\NVAR$
and all $\varphi\in C^2([0,T]\times\overline{\Omega})$:
\begin{equation}\label{eq:consist_rho}
\left[\int_\Omega\rho_{i,h}\,\varphi\,dx\right]_{t=0}^{t=\tau}
=\int_0^\tau\!\int_\Omega
\bigl[\rho_{i,h}\,\partial_t\varphi
+\rho_{i,h}\vel_h\cdot\nabla_x\varphi\bigr]dx\,dt
+\int_0^\tau \mathcal{R}^h_{\rho_i}(t,\varphi)\,dt;
\end{equation}

\item \emph{(Momentum)} For all
$\boldsymbol{\phi}\in C^2([0,T]\times\overline{\Omega};\mathbb{R}^N)$:
\begin{align}\label{eq:consist_m}
\left[\int_\Omega\moment_h\cdot\boldsymbol{\phi}\,dx\right]_{t=0}^{t=\tau}
&=\int_0^\tau\!\int_\Omega
\Bigl[\moment_h\cdot\partial_t\boldsymbol{\phi}
+\frac{\moment_h\otimes\moment_h}{\rho_h}:\nabla_x\boldsymbol{\phi}\nonumber\\
\qquad \qquad &+\pressure(\rho_{1,h},\ldots,\rho_{\NVAR,h},\ent_h)\,
\mathrm{div}_x\boldsymbol{\phi}\Bigr]dx\,dt
+\int_0^\tau \mathcal{R}^h_{\moment}(t,\boldsymbol{\phi})\,dt;
\end{align}

\item \emph{(Entropy inequality)} For all
$\varphi\in C^2([0,T]\times\overline{\Omega})$, $\varphi\geq 0$:
\begin{equation}\label{eq:consist_eta}
\left[\int_\Omega\ent_h\,\varphi\,dx\right]_{t=0}^{t=\tau}
\leq\int_0^\tau\!\int_\Omega
\bigl[\ent_h\,\partial_t\varphi
+(\ent_h\vel_h)\cdot\nabla_x\varphi\bigr]dx\,dt
+\int_0^\tau \mathcal{R}^h_\ent(t,\varphi)\,dt;
\end{equation}

\item \emph{(Energy conservation)}
\begin{equation}\label{eq:consist_E}
\int_\Omega \ener_h(\tau)\,dx=\int_\Omega \ener_h(0)\,dx;
\end{equation}

\item \emph{(Vanishing error)} All consistency errors tend to zero:
\begin{equation}\label{eq:errors_to_zero}
\|\mathcal{R}^h_j\|_{L^1(0,T)}\to 0
\quad\text{as }h\to 0,
\qquad j\in\bigl\{\rho_1,\ldots,\rho_\NVAR,\,\moment,\,\ent\bigr\}.
\end{equation}

\end{enumerate}
\end{theorem}

\begin{proof}
For completeness, we summarize the results from before: 
The consistency errors $\mathcal{R}^h_j$ for $j\in\{\rho_1,\ldots,
\rho_\NVAR,\,\moment,\,\ent\}$ are bounded by terms
$\mathrm{(I)}$--$\mathrm{(III)}$ of decomposition
\eqref{eq:error_decomp}, all of which vanish as $h\to 0$,
terms $\mathrm{(II)}$ and $\mathrm{(III)}$ by the a priori stability
bounds, and term $\mathrm{(I)}$
by the weak BV estimate \eqref{BV_bound}. For the entropy component, the
additional term $\mathrm{(IV)}_\ent$ arising from the numerical
entropy flux is controlled by \eqref{eq:lip_ent_flux} and
\eqref{BV_bound}. Energy conservation \eqref{eq:consist_E}
follows from the conservative structure of the scheme under
periodic boundary conditions; see~\cite{abgrall2022convergence, kuzmin2025consistency} for more details.
\end{proof}
\section{Convergence of the FV scheme inside the DW framework}\label{sec:lim_process}

We prove that the entropy-stable FV scheme 
\eqref{eq:semi_scheme}-\eqref{eq:LF_flux_def} admits a convergent 
subsequence whose limit is a  solution of the 
multicomponent Euler system in the sense of Definition~\ref{def:DW}.  The analysis follows 
\cite{abgrall2022convergence, feireisl2021numerical,feireisl2020convergence, kuzmin2025consistency,yuan2023, lukavcova2023convergence} and is adapted to the multicomponent setting.
Furthermore, we demonstrate a weak-strong uniqueness result and prove strong convergence to strong solutions as the latter exists. \\
We start by providing the following weak convergence result:

\begin{theorem}[Weak convergence]\label{thm:weak_conv}
Let $\mathbf{U}_h=(\rho_{1,h},\ldots,\rho_{\NVAR,h},\moment_h,\eta_h)_{h>0}$
be a family of numerical solutions generated by the scheme
\eqref{eq:semi_scheme}-\eqref{eq:LF_flux_def}. Furthermore, we assume that Assumption~\ref{as_1} and the hypotheses of 
Lemma~\ref{pos_pressure} holds. Then, there exists a subsequence denoted again by $\mathbf{U}_h$ such that 
\begin{subequations}\label{eq:weak_conv}
\begin{align}
\rho_{i,h} &\overset{*}\rightharpoonup \rho_i
&&\text{weakly-$*$ in }
L^\infty(0,T;\,L^1(\Omega)),
\quad i=1,\ldots,\NVAR,
\label{eq:weak_rho_i}\\
\rho_h &\overset{*}\rightharpoonup \rho
&&\text{weakly-$*$ in }
L^\infty(0,T;\,L^{\gamma_{\min}}(\Omega)),
\label{eq:weak_rho_mix}\\
\moment_h &\overset{*}\rightharpoonup \moment
&&\text{weakly-$*$ in }
L^\infty(0,T;\,L^k(\Omega;\mathbb{R}^N)),
\quad k = \tfrac{2\gamma_{\min}}{\gamma_{\min}+1},
\label{eq:weak_m}\\
\eta_h &\overset{*}\rightharpoonup \eta
&&\text{weakly-$*$ in }
L^\infty(0,T;\,L^1(\Omega)),
\label{eq:weak_eta}
\end{align}
\end{subequations}
as $h\to 0$, where $\gamma_{\min}:=\min\{\gamma_1,\ldots,
\gamma_\NVAR\}>1$. Moreover, the limit
$(\rho_1,\ldots,\rho_\NVAR,\moment,\eta,
\mathcal{V}_{t,x},\mathfrak{R},\mathfrak{E})$
is a DW solution of the multicomponent Euler
system in the sense of Definition~\ref{def:DW}, with initial
data $\mathcal{V}_{0,x}=\delta_{(\rho_1^0(x),\ldots,
\rho_\NVAR^0(x),\moment^0(x),\eta^0(x))}$.
\end{theorem}

\begin{proof}
\textbf{Weak convergences and Young measure.}
The convergences \eqref{eq:weak_conv} follow from the Banach-Alaoglu theorem, given the uniform stability bounds. By the fundamental theorem of Young measures~\cite{ball2005version},
the bounded family generates a parametrized probability
measure $\mathcal{V}_{t,x}\in\mathcal{P}(\mathbb{F})$
for a.a.\ $(t,x)\in(0,T)\times\Omega$, with phase space
$\mathbb{F}$ as in \eqref{eq:DW-measure}, such that the
weak limits \eqref{eq:weak_conv} are the first moments:
\begin{equation}\label{eq:YM_moments}
\rho_i = \langle\ym_{t,x};\,\tilde\rho_i\rangle,
\quad
\moment = \langle\ym_{t,x};\,\tilde\moment\rangle,
\quad
\ent = \langle\ym_{t,x};\,\tilde\ent\rangle.
\end{equation}

\medskip
\noindent\textbf{Defect measures.}
The nonlinear momentum flux
$\moment_h\otimes\moment_h/\rho_h+\pressure_h\mathbb{I}$,
bounded in the space $L^\infty(0,T;\,L^1(\Omega;\mathbb{R}^{N\times N}))$,
converges weakly-$*$ in
$L^\infty(0,T;\,\mathcal{M}^+(\overline\Omega;
\mathbb{R}^{N\times N}_{\mathrm{sym}}))$, defining the
Reynolds defect $\mathfrak{R}=\mathfrak{R}_{cd}
+\mathfrak{R}_{od}$. The total energy $\ener_h$ converges
weakly-$*$ in $L^\infty(0,T;\,\mathcal{M}^+(\overline\Omega))$,
defining the energy defect
$\mathfrak{E}=\mathfrak{E}_{cd}+\mathfrak{E}_{od}$.
Non-negativity of oscillation defects $\mathfrak{E}_{od}$ and $\mathfrak{R}_{od}$
follows from Jensen's inequality applied to the convex energy
functional; see~\cite[Theorem~5.4]{feireisl2021numerical}.
The defect compatibility
$\underline{k}\,\mathfrak{E}\leq\mathrm{tr}[\mathfrak{R}]
\leq\bar{k}\,\mathfrak{E}$
follows from~\cite[Lemma~2.1]{feireisl2016dissipative}.

\medskip
\noindent\textbf{Verification of Definition~\ref{def:DW}.}
Using the consistency formulation of
Theorem~\ref{thm:consistency} and passing to the limit via
\eqref{eq:YM_moments}:

\smallskip
\noindent\emph{Continuity equations.}
From \eqref{eq:consist_rho},
the DW continuity identity \eqref{eq:DW-continuity} holds
for each $i=1,\ldots,\NVAR$ independently. No concentration
measure arises due to the linear structure.

\smallskip
\noindent\emph{Momentum equation:}
From \eqref{eq:consist_m}, the DW momentum identity
\eqref{eq:DW-momentum} holds with the Reynolds defect
$\mathfrak{R}$.

\smallskip
\noindent\emph{Entropy balance:}
From \eqref{eq:consist_eta}, the DW entropy balance
\eqref{eq:DW-entropy} holds. The Young measure is retained
in the entropy flux $\langle\ym;\,
\mathbf{1}_{\tilde\rho>0}(\tilde\ent/\tilde\rho)
\tilde\moment\rangle$ since the multicomponent entropy
flux is nonlinear in all species.

\smallskip
\noindent\emph{Energy inequality:}
The scheme conserves total energy exactly under periodic
boundary conditions, cf.\ \eqref{eq:consist_E}. Passing
to the limit via weak lower semicontinuity of the convex
energy functional gives \eqref{eq:DW-energy}, with
$\mathfrak{E}(\tau)\geq 0$ absorbing the possible energy
loss in the limit.

\smallskip
All conditions of Definition~\ref{def:DW} are satisfied,
see~\cite{abgrall2022convergence,feireisl2021numerical} for the
complete argument in the single-component case.
\end{proof}




Next, we  exploit the DW structure established in
Theorem~\ref{thm:weak_conv} to prove that the numerical
solutions converge strongly to the strong solution of
\eqref{eq:multi_component_euler} as long as the strong solutions exist. To this end,  we first demonstrate  weak-strong uniqueness by applying the relative
entropy framework \cite{dafermos2001hyperbolic}, adjusted to the multicomponent setting 
following the single component case from ~\cite{feireisl2016dissipative,
feireisl2020convergence}, see also \cite{gwiazda2020}. 

Let $\bar\con = (\bar\rho_1,\ldots,\bar\rho_\NVAR,
\bar\moment,\bar\ent)$ be a strong solution of
\eqref{eq:multi_component_euler} on $[0,T]\times\Omega$
satisfying
\begin{equation}\label{eq:strong_bounds}
\bar\rho_i > 0\quad (i=1,\ldots,\NVAR),\qquad
\bar\temp > 0,\qquad
\|\bar\con\|_{C^1([0,T]\times\Omega)} \leq M < \infty,
\end{equation}
and set $\bar\rho:=\sum_{i=1}^\NVAR\bar\rho_i$ and
$\bar\vel:=\bar\moment/\bar\rho$. The \emph{relative entropy
density} is defined by
\begin{equation}\label{eq:RE_def}
\relent(\con\,|\,\bar\con)
:= \ent(\con) - \ent(\bar\con)
- \nabla_\con\ent(\bar\con)\cdot(\con-\bar\con)
\geq 0,
\end{equation}
with non-negativity following from convexity of $\ent$.
By Taylor's theorem and the convexity of $\ent$ on compact
subsets of $\mathbb{F}$, the relative entropy satisfies the
coercivity estimate
\begin{equation}\label{eq:coercivity}
c_1\,|\con-\bar\con|^2
\leq \relent(\con\,|\,\bar\con)
\leq c_2\,|\con-\bar\con|^2
\end{equation}
for constants $c_1,c_2>0$ depending only on $M$ and a
compact neighbourhood of $\bar\con$. The \emph{total relative
entropy} incorporating the dissipation defect is,
\begin{equation}\label{eq:RE_total}
\mathcal{H}_h(t)
:= \int_\Omega
\langle\ym_{t,x},\,
\relent(\tilde\con\,|\,\bar\con(t,\cdot))\rangle\,dx
+ \mathfrak{E}(t)
\geq 0.
\end{equation}
Both terms are non-negative: the first by \eqref{eq:RE_def},
the second by the DW energy inequality \eqref{eq:DW-energy}.
When $\ym_{t,x}=\delta_{\bar\con(t,x)}$ and $\mathfrak{E}=0$,
we have $\mathcal{H}_h(t)=0$.

\begin{remark}\label{rem:defect_in_RE}
The inclusion of $\mathfrak{E}$ in \eqref{eq:RE_total} is
essential to close the Gr\"{o}nwall argument in the presence
of oscillations and concentrations, without it the energy
defect contribution cannot be absorbed into $\mathcal{H}_h$.
\end{remark}

\begin{proposition}\label{prop:RE}
Let $(\rho_1,\ldots,\rho_\NVAR,\moment,\ent,\ym_{t,x},
\mathfrak{R},\mathfrak{E})$ be the DW solution from
Theorem~\ref{thm:weak_conv}, and let $\bar\con$ satisfy
\eqref{eq:strong_bounds}. Then for a.a.\ $t\in(0,T)$,
\begin{equation}\label{eq:RE_ineq}
\mathcal{H}_h(t)
\leq \mathcal{H}_h(0)
+ \int_0^t C(s)\,\mathcal{H}_h(s)\,ds
+ \int_0^t \mathcal{R}_h(s)\,ds,
\end{equation}
where $C\in L^1(0,T)$ depends on $\|\nabla_x\bar\con\|_{L^\infty}$, $\|\partial_t\bar\con\|_{L^\infty}$, and the Hessians $\|\ent''\|_{L^\infty}$, $\|\conflux''\|_{L^\infty}$. The consistency residual satisfies $\mathcal{R}_h\to 0$ in $L^1(0,T)$ as $h\to 0$.
\end{proposition}

\begin{proof}
The proof follows~\cite{feireisl2021numerical}, testing the DW conditions
\eqref{eq:DW-continuity}, \eqref{eq:DW-momentum},
\eqref{eq:DW-entropy}, and \eqref{eq:DW-energy} against
test functions derived from $\nabla_\con\ent(\bar\con)$
and comparing with the strong equations for $\bar\con$,
all linear terms in $(\con-\bar\con)$ cancel by the chain
rule identities
\[
\nabla_x(\nabla_\con\ent(\bar\con))
= \ent''(\bar\con)\nabla_x\bar\con,
\qquad
\partial_t(\nabla_\con\ent(\bar\con))
= \ent''(\bar\con)\partial_t\bar\con.
\]
Three additional contributions arise. First, the consistency
residuals from Theorem~\ref{thm:consistency} produce
$\mathcal{R}_h\to 0$ in $L^1(0,T)$. Second, the Reynolds
defect $\mathfrak{R}$ in \eqref{eq:DW-momentum} contributes
\[
\left|\int_\Omega\mathfrak{R}:
\nabla_x(\nabla_\moment\ent(\bar\con))\,dx\right|
\leq C\|\bar\con\|_{C^1}
\int_\Omega d|\mathfrak{R}|
\leq C(t)\,\mathcal{H}_h(t),
\]
using the defect compatibility \eqref{eq:DW-defect}. Third,
the energy defect $\mathfrak{E}$ in \eqref{eq:DW-energy}
contributes $C(t)\,\mathfrak{E}(t)\leq C(t)\,\mathcal{H}_h(t)$
by \eqref{eq:strong_bounds} and \eqref{eq:RE_total}.
Collecting all terms yields \eqref{eq:RE_ineq},
see~\cite{feireisl2016dissipative} for the extension to
the multicomponent pressure
$\pressure(\rho_1,\ldots,\rho_\NVAR,\ent)$.
\end{proof}

We now proceed to establish the weak-strong uniqueness property for completeness:
\begin{theorem}[Weak-strong uniqueness]\label{thm:WSU}
Let $\bar\con\in C^1([0,T]\times\Omega)$ be a classical
solution of \eqref{eq:multi_component_euler} satisfying
\eqref{eq:strong_bounds}, and let
$(\rho_1,\ldots,\rho_\NVAR,\moment,\ent,\ym_{t,x},
\mathfrak{R},\mathfrak{E})$ be a DW solution in the sense
of Definition~\ref{def:DW} with the same initial data
\[
\ym_{0,x}
= \delta_{(\bar\rho_1(0,x),\ldots,
\bar\rho_\NVAR(0,x),\bar\moment(0,x),\bar\eta(0,x))}.
\]
Then $\ym_{t,x}=\delta_{\bar\con(t,x)}$,
$\mathfrak{R}(t)=0$, and $\mathfrak{E}(t)=0$
for a.a.\ $t\in(0,T)$.
\end{theorem}

\begin{proof}
Since $\ym_{0,x}=\delta_{\bar\con(0,x)}$, we have
$\mathcal{H}_h(0)=0$. Since $\bar\con$ is a classical
solution, it satisfies the equations exactly, so
$\mathcal{R}_h=0$ in \eqref{eq:RE_ineq}. Gr\"{o}nwall's
inequality gives $\mathcal{H}_h(t)=0$ for all $t\in[0,T]$.
Since both terms in \eqref{eq:RE_total} are non-negative,
\[
\mathfrak{E}(t)=0
\qquad\text{and}\qquad
\int_\Omega\langle\ym_{t,x},\,
\relent(\tilde\con|\bar\con)\rangle\,dx = 0
\quad\text{for a.a. }t.
\]
The coercivity \eqref{eq:coercivity} forces
$\ym_{t,x}=\delta_{\bar\con(t,x)}$
for a.e.\ $(t,x)$. Finally, $\mathfrak{R}(t)=0$
follows from \eqref{eq:DW-defect} and $\mathfrak{E}(t)=0$,
$\mathrm{tr}[\mathfrak{R}]\leq\bar{k}\,\mathfrak{E}=0$
and $\mathfrak{R}$ being positive semi-definite.
\end{proof}

The preceding results allow us to conclude the following strong convergence theorem:

\begin{theorem}[Strong convergence]\label{thm:strong_conv}
Under the hypotheses of Theorem~\ref{thm:weak_conv} and
Proposition~\ref{prop:RE}, suppose the multicomponent Euler 
system admits a classical solution
$\bar\con\in C^1([0,T]\times\Omega)$ satisfying
\eqref{eq:strong_bounds}, and that the initial data are
well-prepared: 
\begin{equation}\label{eq:wellprepared}
\mathcal{H}_h(0)
= \int_\Omega
\relent\bigl(\con_h(0,\cdot)\,|\,\bar\con(0,\cdot)\bigr)\,dx
\to 0
\quad\text{as }h\to 0.
\end{equation}
Then, as $h\to 0$ along the subsequence of
Theorem~\ref{thm:weak_conv}:
\begin{subequations}\label{eq:strong_conv}
\begin{align}
\rho_{i,h} &\to \bar\rho_i
&&\text{strongly in }L^1((0,T)\times\Omega),
\quad i=1,\ldots,\NVAR,
\label{eq:conv_rho}\\
\rho_h &\to \bar\rho
&&\text{strongly in }L^q((0,T)\times\Omega),
\quad 1\leq q < \gamma_{\min},
\label{eq:conv_rho_mix}\\
\moment_h &\to \bar\moment
&&\text{strongly in }L^q((0,T)\times\Omega;\mathbb{R}^N),
\quad 1\leq q < k,
\label{eq:conv_m}\\
\ent_h &\to \bar\ent
&&\text{strongly in }L^1((0,T)\times\Omega),
\label{eq:conv_eta}
\end{align}
\end{subequations}
where $k = 2\gamma_{\min}/(\gamma_{\min}+1)$.
\end{theorem}

\begin{proof}
Applying Gr\"{o}nwall to \eqref{eq:RE_ineq} with
\eqref{eq:wellprepared} and $\mathcal{R}_h\to 0$:
\[
\sup_{t\in[0,T]}\mathcal{H}_h(t)
\leq\Bigl(\mathcal{H}_h(0)
+\int_0^T\mathcal{R}_h(s)\,ds\Bigr)
\exp\!\Bigl(\int_0^T C(s)\,ds\Bigr)
\to 0.
\]
By Theorem~\ref{thm:WSU}, $\ym_{t,x}=\delta_{\bar\con(t,x)}$
and $\mathfrak{E}=0$. The coercivity \eqref{eq:coercivity}
then gives
\[
\int_\Omega|\con_h - \bar\con|^2\,dx
\leq \frac{1}{c_1}\mathcal{H}_h(t)
\to 0 \quad\text{uniformly in }t,
\]
so $\con_h\to\bar\con$ strongly in
$L^\infty(0,T;\,L^2(\Omega))$ locally near $\bar\con$.
Since $\bar\con\in C^1$ is uniformly bounded and
$\|\con_h - \bar\con\|_{L^\infty(L^2)}\to 0$,
strong $L^1((0,T)\times\Omega)$ convergence of $\rho_{i,h}$
and $\eta_h$ follows by the Cauchy-Schwarz inequality.
For $\rho_h\in L^\infty(0,T;\,L^{\gamma_{\min}}(\Omega))$
and $\moment_h\in L^\infty(0,T;\,L^k(\Omega;\mathbb{R}^N))$,
interpolation between the $L^2$ convergence and the
stability bounds of Lemma~\ref{pos_pressure} yields strong
convergence in $L^q$ for $1\leq q < \gamma_{\min}$ and
$1\leq q < r$ respectively (see~\cite{feireisl2021numerical} for details). The $'\NVAR'$ partial densities
$\rho_{i,h}$ converge separately since each satisfies an
independent continuity equation \eqref{eq:DW-continuity}.
\end{proof}


\section{Numerical Simulation}\label{sec:numerics}

In the following section, we present numerical experiments that validate our theoretical findings. Rather than analyzing the FV scheme \eqref{eq:semi_scheme}-\eqref{eq:LF_flux_def} alone, we extend the numerical study to structure-preserving discontinuous Galerkin spectral element methods (DGSEM) for the multicomponent Euler equations, as described in \cite{renac2021entropy}. Convergence analyses of DGSEM for the classical Euler equations and for a stochastic version of the method were carried out in \cite{lukavcova2023convergence} and \cite{breit2025discontinuous}, respectively. 
It is expected that the convergence results transfer as well to the multicomponent setting directly, as will be demonstrated in the section. 
All simulations are implemented in \texttt{Trixi.jl}~\cite{ranocha2021adaptive, schlottkelakemper2025trixi} where further tests on multicomponent Euler equations can be found. We provide numerical verification of the theoretical convergence results.

\subsection{Smooth convergence test}\label{sec:smooth_test} 
We verify the strong convergence result of 
Theorem~\ref{thm:strong_conv} numerically following 
the approach of~\cite{lukavcova2023convergence} and use the provided test inside \texttt{Trixi.jl}. We 
employ the method of manufactured solutions on 
$\Omega = [-1,1]^2$ with periodic boundary conditions. 
The spatial discretization uses the FV setting ($polydeg=0$) and the nodal DGSEM with polynomial degree $(polydeg=1)$ with 
 the local Lax-Friedrichs surface flux. Time integration is 
performed with the  
Runge-Kutta scheme of Carpenter and 
Kennedy~\cite{Kennedy2003} at $\mathrm{CFL}=0.5$.

The two-component Euler system 
\eqref{eq:multi_component_euler} with parameters 
$\gamma_1 = \gamma_2 = 1.4$ and $r_1 = r_2 = 0.4$ is 
supplemented by a smooth source term consistent with the 
exact solution
\begin{equation}\label{eq:exact_sol}
\rho(t,x) = c + A\sin\!\bigl(\pi(x_1+x_2-t)\bigr),
\quad
\rho_1 = \tfrac{1}{3}\rho,
\quad
\rho_2 = \tfrac{2}{3}\rho,
\quad
\vel = (1,1)^T,
\quad
\ener = \rho,
\end{equation}
with $c = 2$ and $A = 0.1$. Table~\ref{tab:eoc} reports 
the $L^2$ errors and experimental orders of convergence 
(EOC) for all conserved variables at time $t = 0.4$ under 
successive mesh refinements, confirming first-order 
convergence for $polydeg = 0$. For comparison, results for 
$polydeg = 1$ are also included, showing the expected 
second-order rate. Similar results for higher-order $polydeg>1$ can be expected as shown in \cite{renac2021entropy} or demonstrated inside the \texttt{Trixi.jl} environment \cite{schlottkelakemper2025trixi}.
\begin{table}[htbp]
\centering
\caption{$L^2$ errors and experimental orders of convergence 
(EOC) for the two-component compressible Euler system, 
smooth convergence test on $\Omega = [-1,1]^2$, $t = 0.4$, 
$\gamma_1 = \gamma_2 = 1.4$, $r_1 = r_2 = 0.4$.}
\label{tab:eoc}
\setlength{\tabcolsep}{4pt}
\begin{tabular}{c|cc|cc|cc|cc|cc}
\hline
& \multicolumn{2}{c|}{$\rho_{1,h}$} 
& \multicolumn{2}{c|}{$\rho_{2,h}$}
& \multicolumn{2}{c|}{$\rho_h \vel_{1,h}$}
& \multicolumn{2}{c|}{$\rho_h \vel_{2,h}$}
& \multicolumn{2}{c}{$\ener_h$} \\
$N$ & error & EOC & error & EOC 
    & error & EOC & error & EOC 
    & error & EOC \\
\hline
\multicolumn{11}{c}{$polydeg = 0$ \quad (FV scheme, 
first order)} \\
\hline
$16^2$  & 7.08e-03 & --   & 1.42e-02 & --   
        & 2.33e-02 & --   & 2.33e-02 & --   
        & 1.07e-01 & --   \\
$32^2$  & 3.93e-03 & 0.85 & 7.87e-03 & 0.85 
        & 1.33e-02 & 0.81 & 1.33e-02 & 0.81 
        & 6.12e-02 & 0.80 \\
$64^2$  & 2.10e-03 & 0.91 & 4.19e-03 & 0.91 
        & 7.29e-03 & 0.87 & 7.29e-03 & 0.87 
        & 3.32e-02 & 0.89 \\
$128^2$ & 1.09e-03 & 0.95 & 2.18e-03 & 0.95 
        & 3.84e-03 & 0.92 & 3.84e-03 & 0.92 
        & 1.73e-02 & 0.94 \\
\hline
mean    &          & 0.90 &          & 0.90 
        &          & 0.87 &          & 0.87 
        &          & 0.88 \\
\hline
\multicolumn{11}{c}{$polydeg = 1$ \quad (second order)} \\
\hline
$16^2$  & 5.44e-04 & --   & 1.09e-03 & --   
        & 3.73e-03 & --   & 3.73e-03 & --   
        & 1.16e-02 & --   \\
$32^2$  & 1.38e-04 & 1.98 & 2.75e-04 & 1.98 
        & 1.00e-03 & 1.90 & 1.00e-03 & 1.90 
        & 2.99e-03 & 1.96 \\
$64^2$  & 3.48e-05 & 1.98 & 6.96e-05 & 1.98 
        & 2.58e-04 & 1.96 & 2.58e-04 & 1.96 
        & 7.52e-04 & 1.99 \\
$128^2$ & 8.77e-06 & 1.99 & 1.75e-05 & 1.99 
        & 6.52e-05 & 1.98 & 6.52e-05 & 1.98 
        & 1.88e-04 & 2.00 \\
\hline
mean    &          & 1.98 &          & 1.98 
        &          & 1.95 &          & 1.95 
        &          & 1.98 \\
\hline
\end{tabular}
\end{table}

\subsection{Kelvin-Helmholtz instability}\label{sec:khi}
To demonstrate the convergence properties inside the DW framework, following the classical Kelvin--Helmholtz instability (KHI) test 
described in~\cite{lukavcova2023convergence, feireisl2021numerical}, 
we design an analogous experiment for the multicomponent setting.
The computational 
domain is $\Omega = [0,1]^2$ with periodic boundary 
conditions and $\NVAR = 2$ species with 
$\gamma_1 = \gamma_2 = 1.4$ and $r_1 = r_2 = 1.0$.

The initial condition consists of four horizontal fluid 
layers separated by three perturbed interfaces located near 
$J_1 = 0.25$, $J_2 = 0.50$, and $J_3 = 0.75$. Each 
interface $I_j$ is perturbed by a superposition of ten 
Fourier modes:
\begin{equation}\label{eq:interface}
I_j(x_1) = J_j + \varepsilon\sum_{i=1}^{10}
a_j^{(i)}\cos\!\bigl(b_j^{(i)} + 2\pi i\,x_1\bigr),
\qquad j = 1, 2, 3,
\end{equation}
with amplitude $\varepsilon = 0.01$ and random 
coefficients $a_j^{(i)}$, $b_j^{(i)}$ normalized so 
that $\sum_{i=1}^{10} a_j^{(i)} = 1$. The partial 
densities and horizontal velocity are assigned by layer:
\begin{equation}\label{eq:khi_layers}
(\rho_1, \rho_2, \vel_1) = 
\begin{cases}
(0.8,\ 0.2,\ -0.5) & x_2 \leq I_1, \\
(0.2,\ 1.8,\ +0.5) & I_1 < x_2 \leq I_2, \\
(1.8,\ 0.2,\ -0.5) & I_2 < x_2 \leq I_3, \\
(0.2,\ 0.8,\ +0.5) & x_2 > I_3,
\end{cases}
\end{equation}
with $\vel_2 = 0$ and uniform pressure $\pressure = 2.5$ throughout 
the domain, ensuring pressure equilibrium across all 
interfaces.
To quantify convergence, we follow the framework 
of~\cite{lukavcova2023convergence} 
and compute two error quantities. Let $\con^{h_n}$ denote 
the numerical solution on a mesh with $n \times n$ cells 
and let $\con^{h_N}$ denote a reference solution on the 
finest mesh with $N \times N$ cells. The classical 
grid-refinement error and the Ces\`aro average error are then 
defined by
\begin{equation}\label{eq:E1E2}
E_1(n) := \bigl\|\con^{h_n} - \con^{h_N}
\bigr\|_{L^1((0,T)\times\Omega)},
\qquad
E_2(n) := \left\|\widetilde{\con}^{h_n} 
- \widetilde{\con}^{h_N}
\right\|_{L^1((0,T)\times\Omega)},
\end{equation}
where the Ces\`aro average is defined by
$
\widetilde{\con}^{h_n} := \frac{1}{n}
\sum_{k=1}^{n} \con^{h_k}.
$
It is expected that we will not observe grid convergence similar to the Euler equations \cite{lukavcova2023convergence} due to rising oscillations on fine-scales, but convergence using  Ces\`aro averages \cite{feireisl2021numerical}.

Figures~\ref{fig:khi_p0}-\ref{fig:khi_p1} display numerical solutions of the partial density $\rho_{1,h}$ at time $t = 2$ on meshes of size $n = 512, 1024, 2048$, computed using the entropy-stable FV scheme \eqref{eq:semi_scheme}-\eqref{eq:LF_flux_def}. For the first-order scheme (Figure~\ref{fig:khi_p0}), we employ a pure 
finite volume discretization ($polydeg = 0$) with the local Lax-Friedrichs surface flux. For the second-order scheme (Figure~\ref{fig:khi_p1}), we use a nodal DGSEM with $polydeg = 1$, combining flux differencing \cite{fisher2013discretely, gassner2016split, oeffner2023} with the entropy-conservative, kinetic-energy-preserving volume flux~\cite{ranocha2018comparison} adapted to the multicomponent setting and the local Lax-Friedrichs surface flux. A Zhang-Shu positivity-preserving limiter~\cite{zhang2010positivity} is applied to the partial densities $\rho_1$, $\rho_2$ and pressure $\pressure$ at each stage of the third-order strong stability preserving Runge-Kutta time integrator (SSPRK33)at $\mathrm{CFL} = 0.8$.

Table~\ref{tab:khi} reports the errors $E_1$ and $E_2$ for successive mesh refinements. Consistent with the theoretical prediction, $E_1$ does not decrease monotonically, the individual numerical solutions exhibit different fine-scale structures at each mesh level due to the sensitivity 
of the KH. In contrast, $E_2$ decreases confirming the convergence of Ces\`aro averages to a DW solution of the multicomponent Euler system \eqref{eq:multi_component_euler}.

\begin{table}[htbp]
\centering
\caption{KHI test at $t = 2$, reference solution on 
$2048\times 2048$ mesh. Classical error $E_1$ and 
Ces\`aro average error $E_2$ with experimental orders 
of convergence for $polydeg = 0$ and $polydeg = 1$.}
\label{tab:khi}
\small
\setlength{\tabcolsep}{4pt}

\medskip
\noindent\textit{Variable: $\rho_{1,h}$}

\begin{tabular}{c|cc|cc|cc|cc}
\hline
& \multicolumn{4}{c|}{$p = 0$ (FV)}
& \multicolumn{4}{c}{$p = 1$ (DGSEM)} \\
$n$ & $E_1$ & EOC & $E_2$ & EOC 
    & $E_1$ & EOC & $E_2$ & EOC \\
\hline
$64$
& 3.40e-01 & --   & 1.82e-01 & --
& 5.08e-01 & --   & 3.08e-01 & -- \\
$128$
& 3.05e-01 & 0.16 & 1.48e-01 & 0.30
& 4.63e-01 & 0.13 & 1.90e-01 & 0.70 \\
$256$
& 2.67e-01 & 0.19 & 1.07e-01 & 0.47
& 3.99e-01 & 0.21 & 1.29e-01 & 0.56 \\
$512$
& 2.40e-01 & 0.15 & 7.24e-02 & 0.56
& 4.60e-01 & -0.20 & 9.73e-02 & 0.41 \\
$1024$
& 1.90e-01 & 0.34 & 4.24e-02 & 0.77
& 4.56e-01 & 0.01 & 6.43e-02 & 0.60 \\
\hline
\end{tabular}

\medskip
\noindent\textit{Variable: $\rho_{2,h}$}

\begin{tabular}{c|cc|cc|cc|cc}
\hline
& \multicolumn{4}{c|}{$p = 0$}
& \multicolumn{4}{c}{$p = 1$} \\
$n$ & $E_1$ & EOC & $E_2$ & EOC 
    & $E_1$ & EOC & $E_2$ & EOC \\
\hline
$64$
& 3.46e-01 & --   & 1.90e-01 & --
& 4.77e-01 & --   & 2.83e-01 & -- \\
$128$
& 2.99e-01 & 0.21 & 1.55e-01 & 0.29
& 5.11e-01 & -0.10 & 1.94e-01 & 0.54 \\
$256$
& 2.52e-01 & 0.25 & 1.12e-01 & 0.47
& 3.95e-01 & 0.37 & 1.26e-01 & 0.62 \\
$512$
& 2.17e-01 & 0.22 & 7.25e-02 & 0.63
& 4.52e-01 & -0.19 & 9.40e-02 & 0.43 \\
$1024$
& 1.66e-01 & 0.39 & 4.04e-02 & 0.84
& 4.72e-01 & -0.06 & 6.59e-02 & 0.51 \\
\hline
\end{tabular}

\medskip
\noindent\textit{Variable: $\rho_h \vel_{1,h}$}

\begin{tabular}{c|cc|cc|cc|cc}
\hline
& \multicolumn{4}{c|}{$p = 0$}
& \multicolumn{4}{c}{$p = 1$} \\
$n$ & $E_1$ & EOC & $E_2$ & EOC 
    & $E_1$ & EOC & $E_2$ & EOC \\
\hline
$64$
& 4.47e-01 & --   & 2.61e-01 & --
& 3.75e-01 & --   & 2.39e-01 & -- \\
$128$
& 3.70e-01 & 0.27 & 2.21e-01 & 0.24
& 4.62e-01 & -0.30 & 1.66e-01 & 0.53 \\
$256$
& 2.37e-01 & 0.64 & 1.60e-01 & 0.46
& 3.45e-01 & 0.42 & 1.19e-01 & 0.47 \\
$512$
& 1.29e-01 & 0.88 & 9.52e-02 & 0.75
& 3.52e-01 & -0.03 & 8.39e-02 & 0.51 \\
$1024$
& 9.89e-02 & 0.38 & 3.88e-02 & 1.30
& 3.35e-01 & 0.07 & 5.00e-02 & 0.75 \\
\hline
\end{tabular}

\medskip
\noindent\textit{Variable: $\ener_h$}

\begin{tabular}{c|cc|cc|cc|cc}
\hline
& \multicolumn{4}{c|}{$p = 0$}
& \multicolumn{4}{c}{$p = 1$} \\
$n$ & $E_1$ & EOC & $E_2$ & EOC 
    & $E_1$ & EOC & $E_2$ & EOC \\
\hline
$64$
& 2.10e-01 & --   & 5.86e-02 & --
& 4.95e-01 & --   & 2.97e-01 & -- \\
$128$
& 2.06e-01 & 0.03 & 5.58e-02 & 0.07
& 5.53e-01 & -0.16 & 2.28e-01 & 0.38 \\
$256$
& 1.97e-01 & 0.06 & 5.11e-02 & 0.13
& 4.34e-01 & 0.35 & 1.48e-01 & 0.62 \\
$512$
& 1.85e-01 & 0.09 & 4.49e-02 & 0.19
& 4.51e-01 & -0.06 & 1.00e-01 & 0.56 \\
$1024$
& 1.39e-01 & 0.41 & 3.10e-02 & 0.53
& 4.49e-01 & 0.01 & 6.32e-02 & 0.67 \\
\hline
\end{tabular}
\end{table}

\begin{figure}[htbp]
\centering
\includegraphics[width=0.32\textwidth]{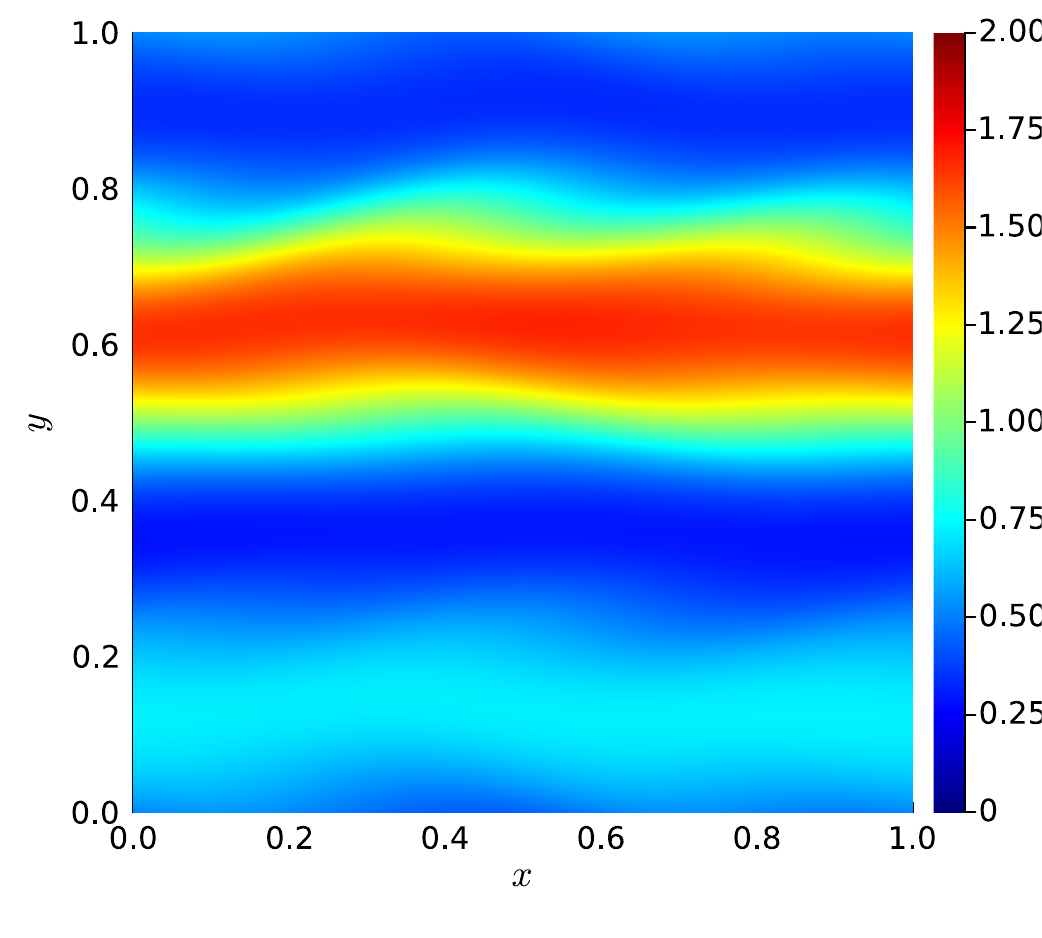}
\includegraphics[width=0.32\textwidth]{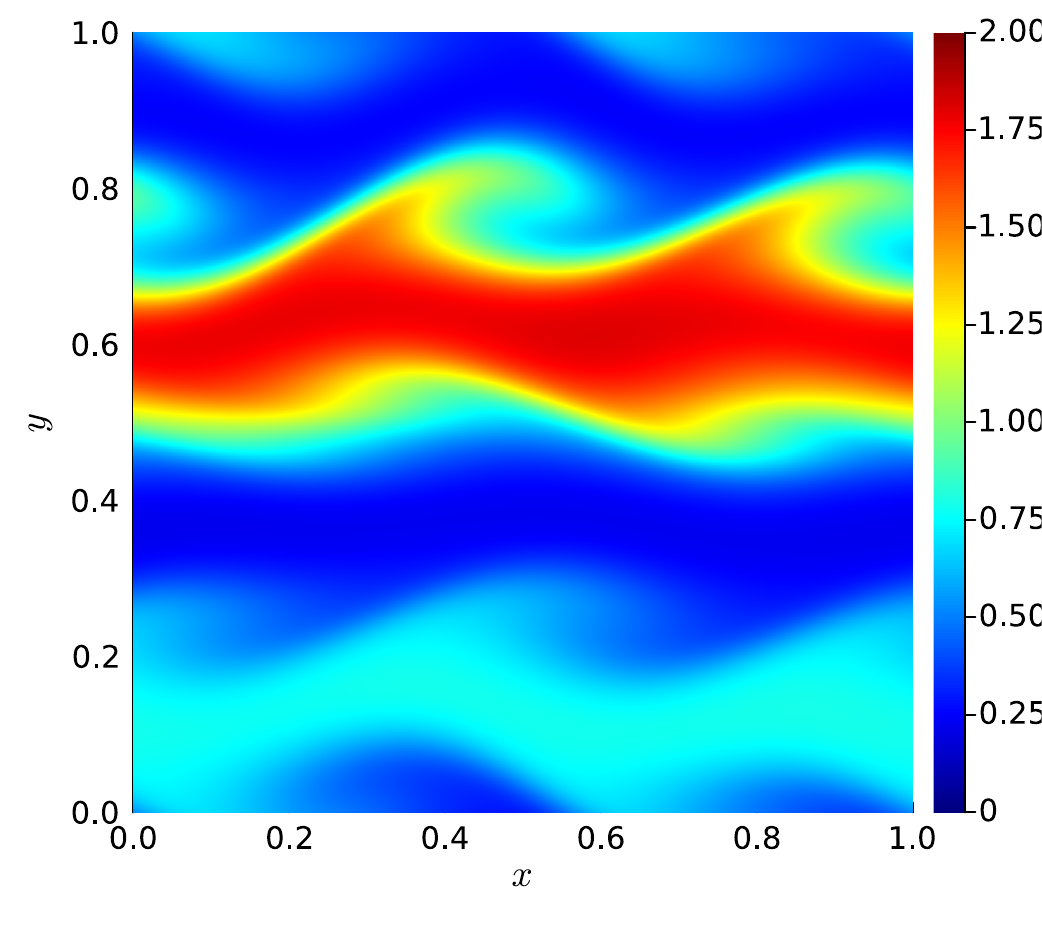}
\includegraphics[width=0.32\textwidth]{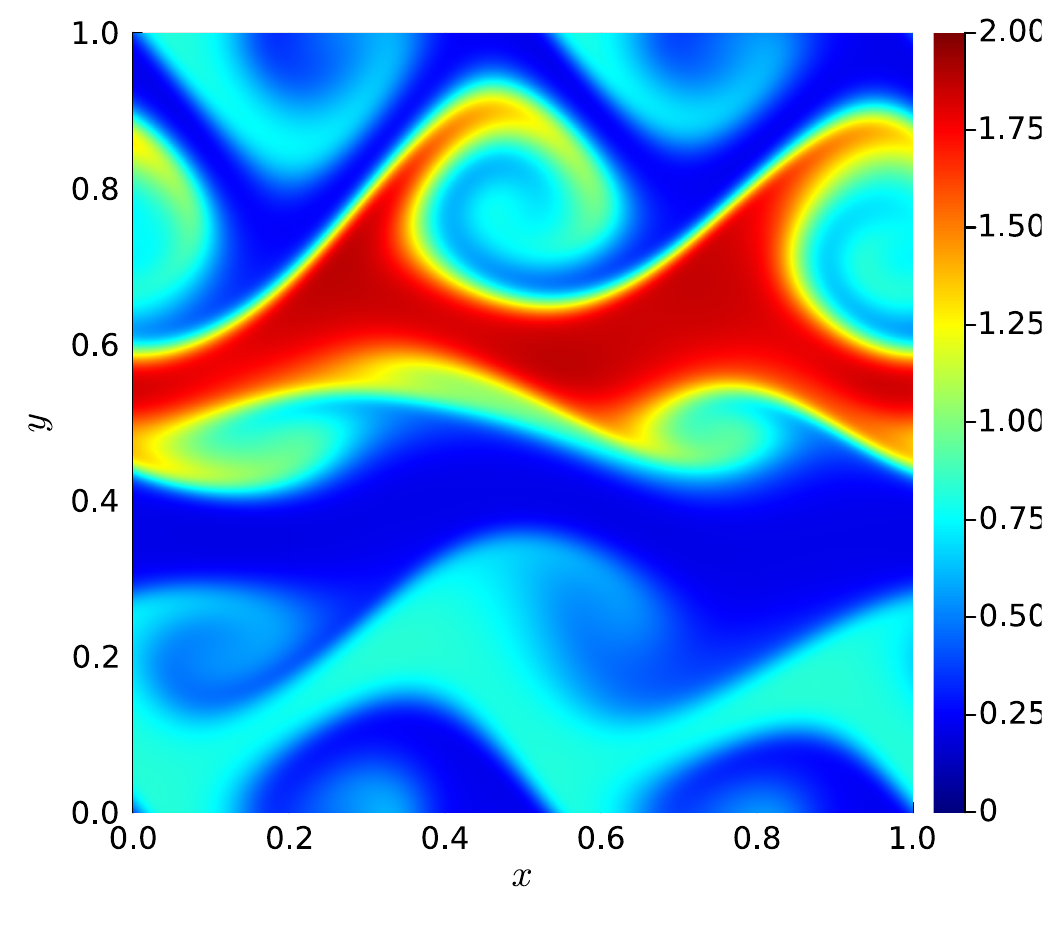}\\[4pt]
\includegraphics[width=0.32\textwidth]{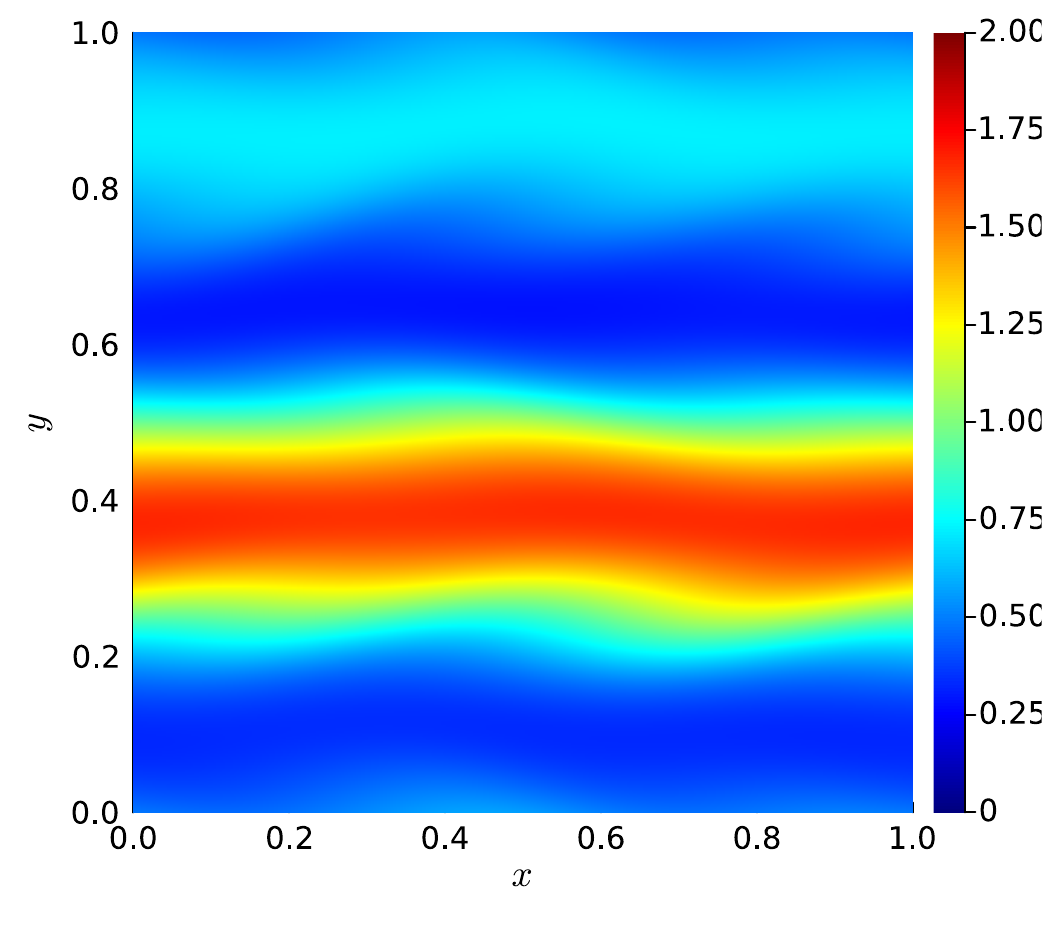}
\includegraphics[width=0.32\textwidth]{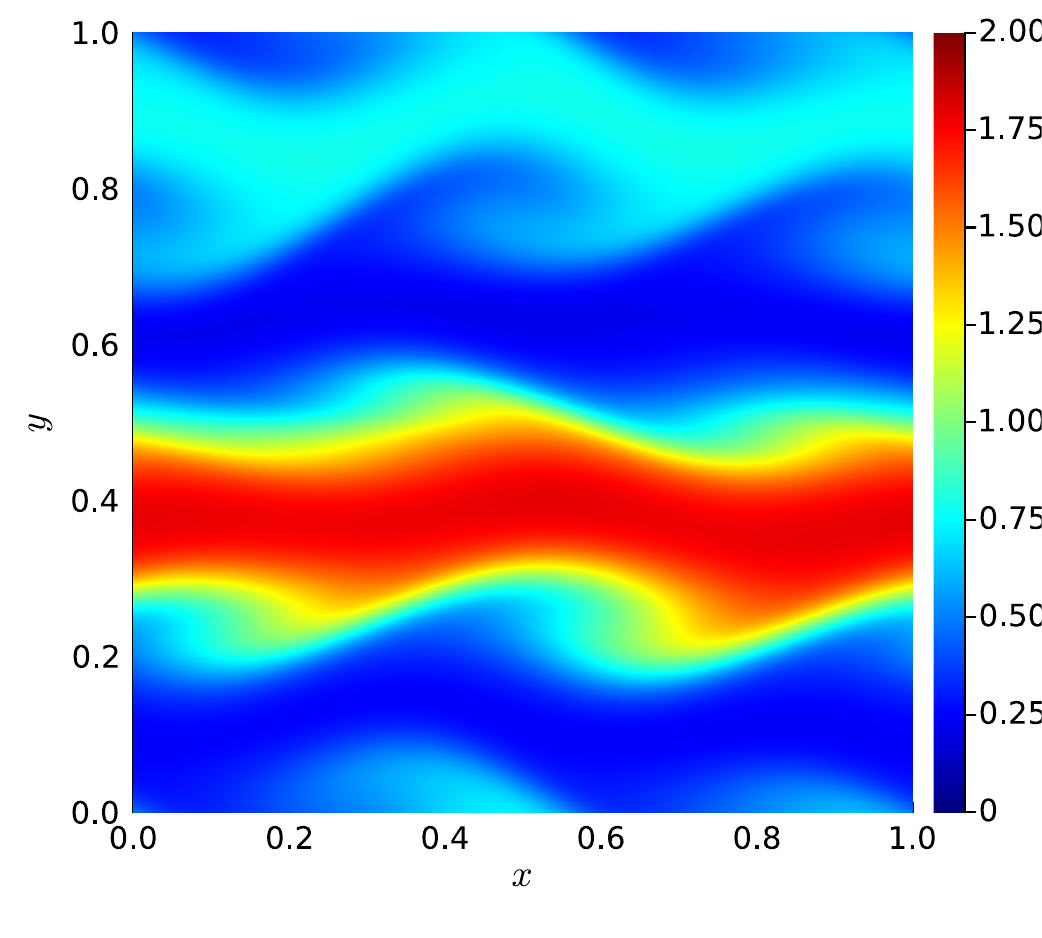}
\includegraphics[width=0.32\textwidth]{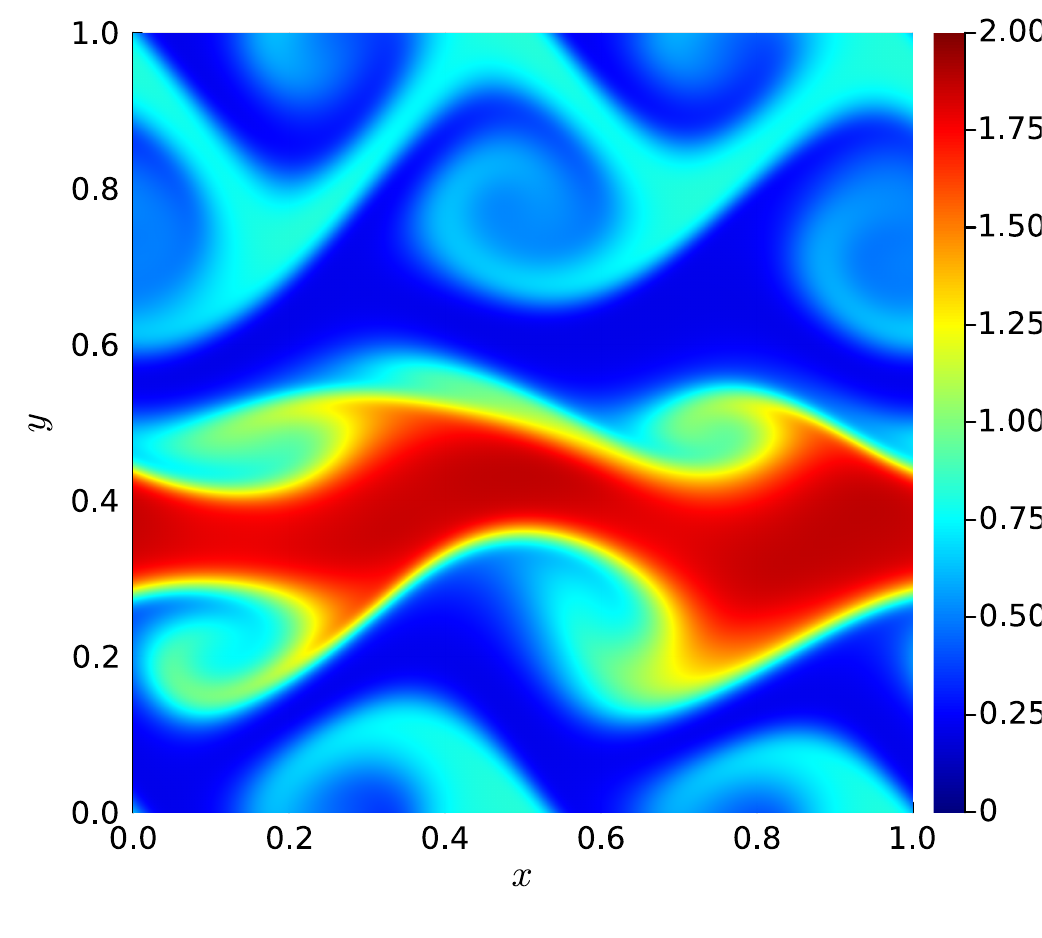}\\[4pt]
\includegraphics[width=0.32\textwidth]{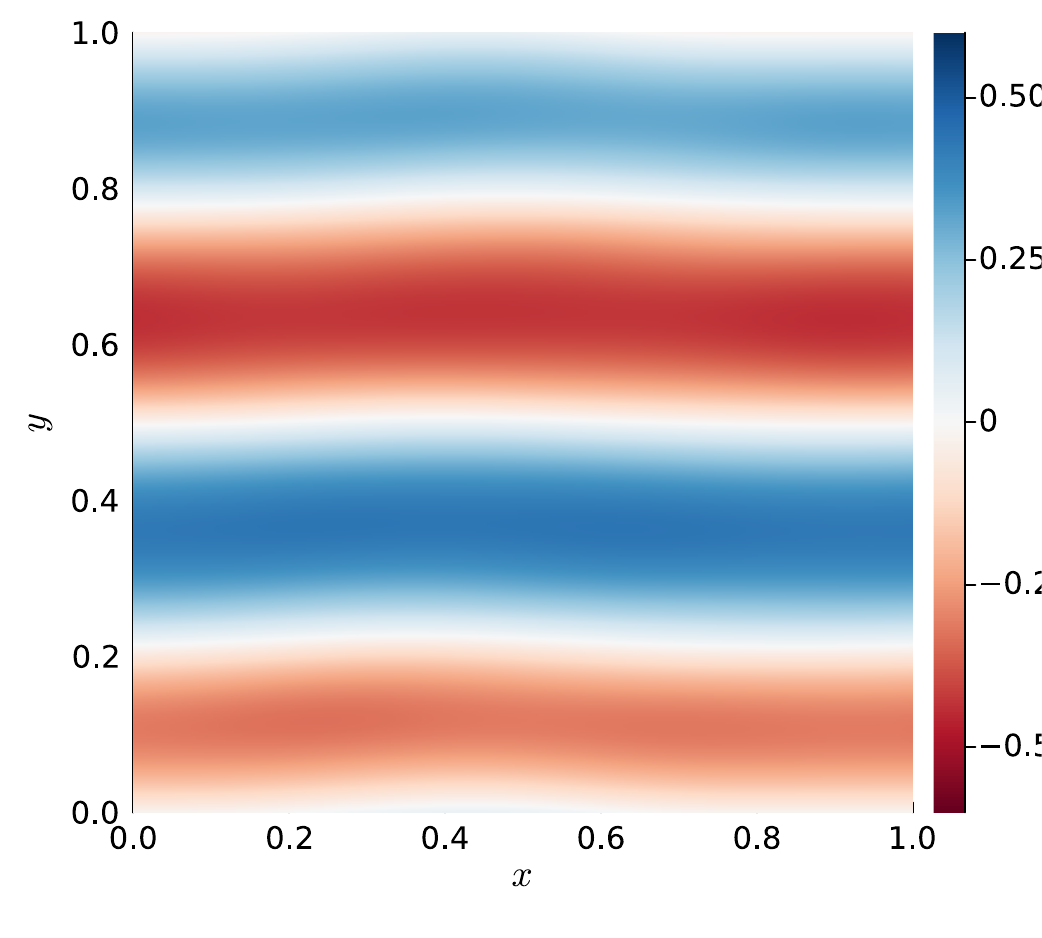}
\includegraphics[width=0.32\textwidth]{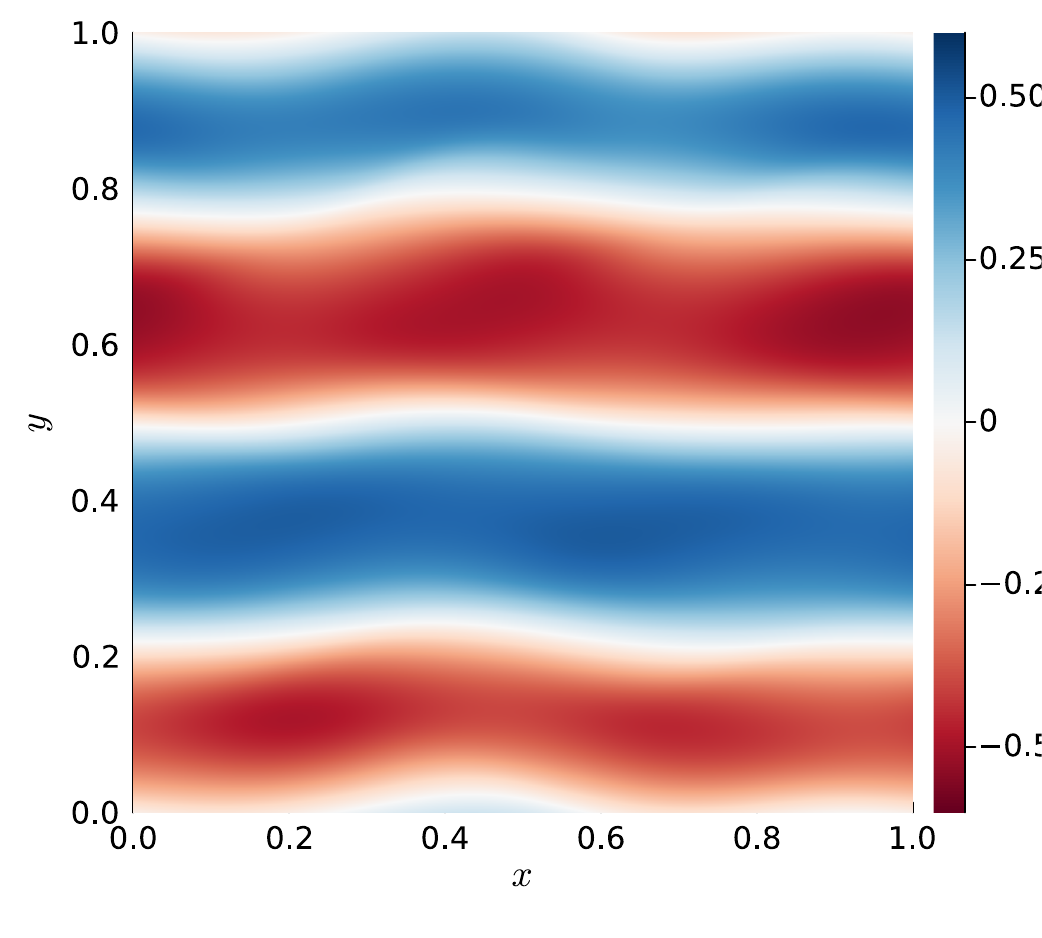}
\includegraphics[width=0.32\textwidth]{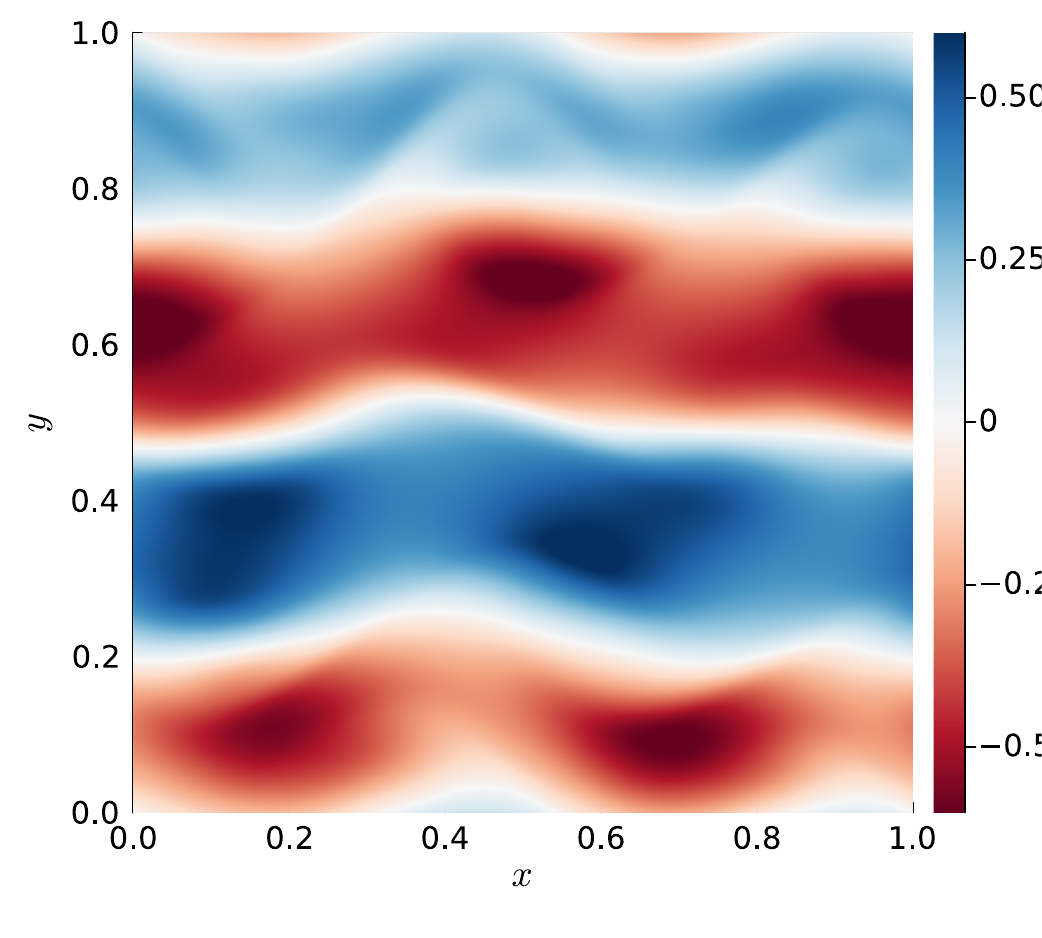}
\caption{KHI test at $t = 2$, $polydeg = 0$. 
Pseudocolor plots of $\rho_{1,h}$ (top row), 
$\rho_{2,h}$ (middle row), and $\vel_{1,h}$ (bottom row) 
on meshes $512^2$, $1024^2$, $2048^2$ (left to right).}
\label{fig:khi_p0}
\end{figure}

\begin{figure}[htbp]
\centering
\includegraphics[width=0.32\textwidth]{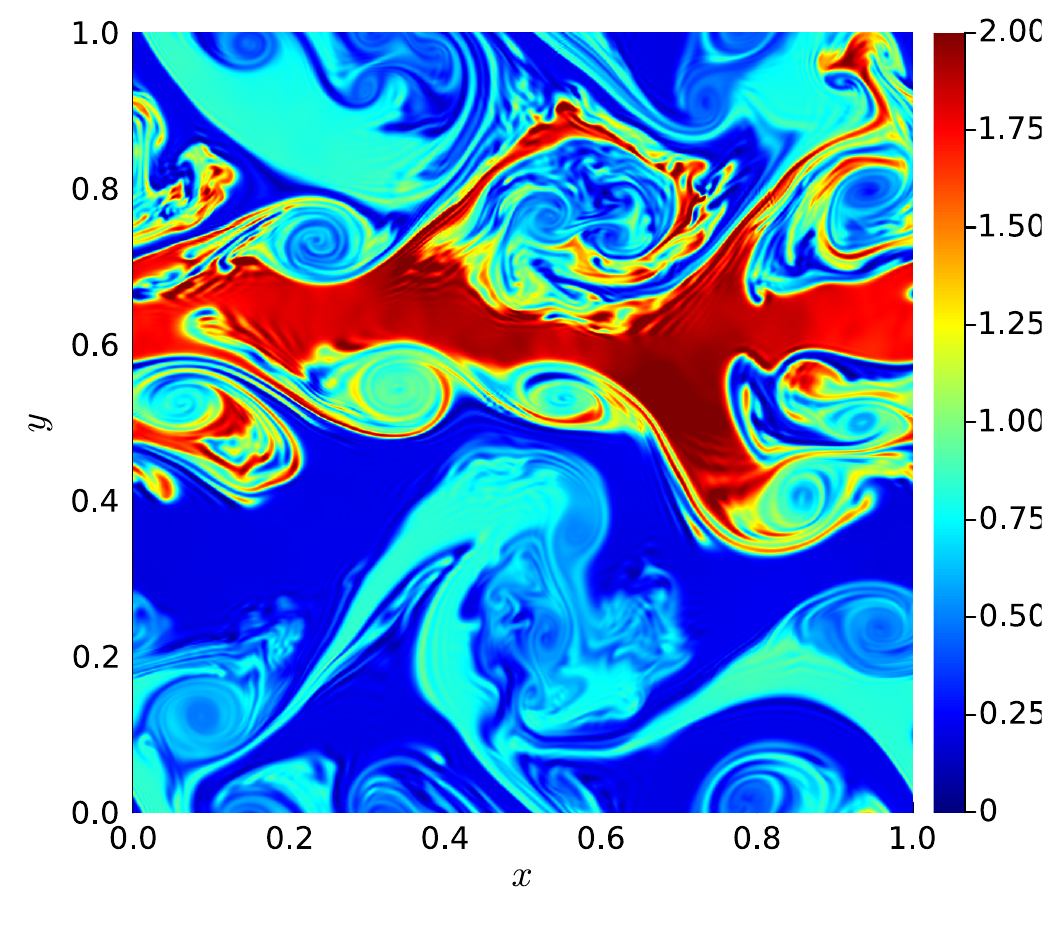}
\includegraphics[width=0.32\textwidth]{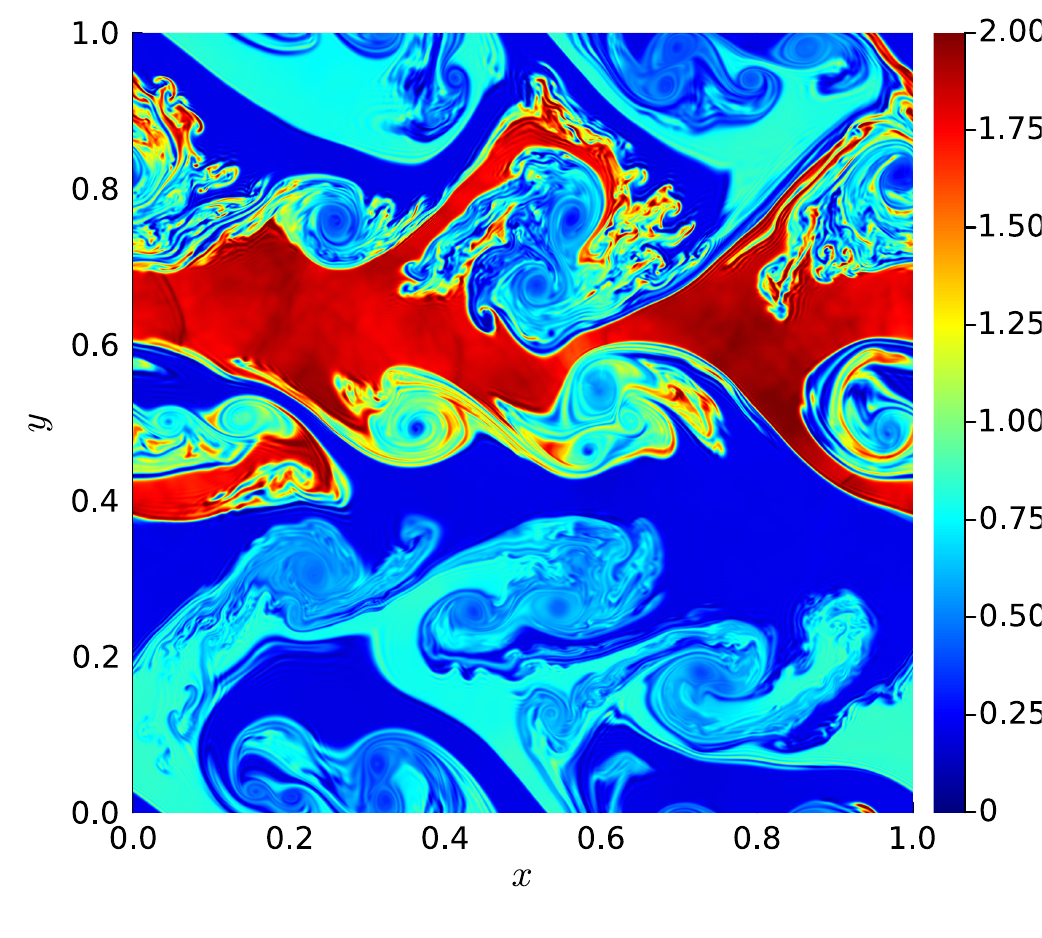}
\includegraphics[width=0.32\textwidth]{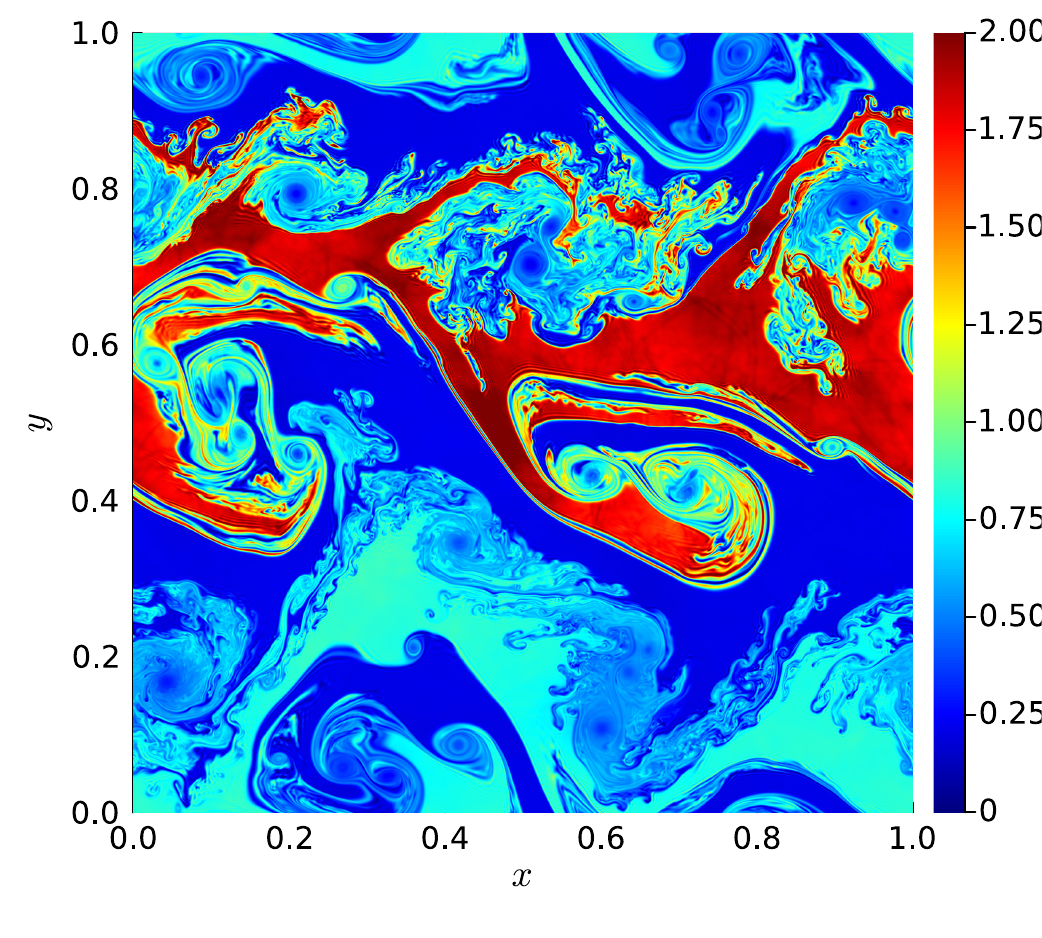}\\[4pt]
\includegraphics[width=0.32\textwidth]{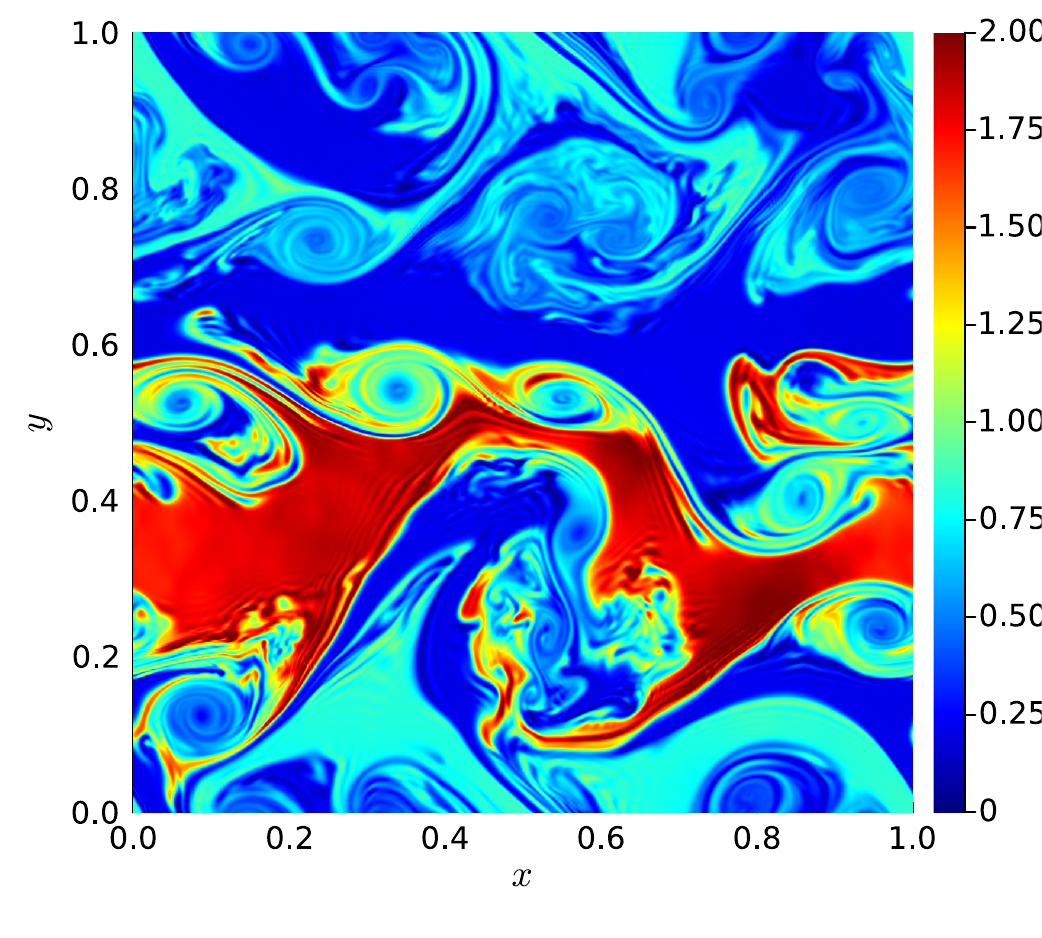}
\includegraphics[width=0.32\textwidth]{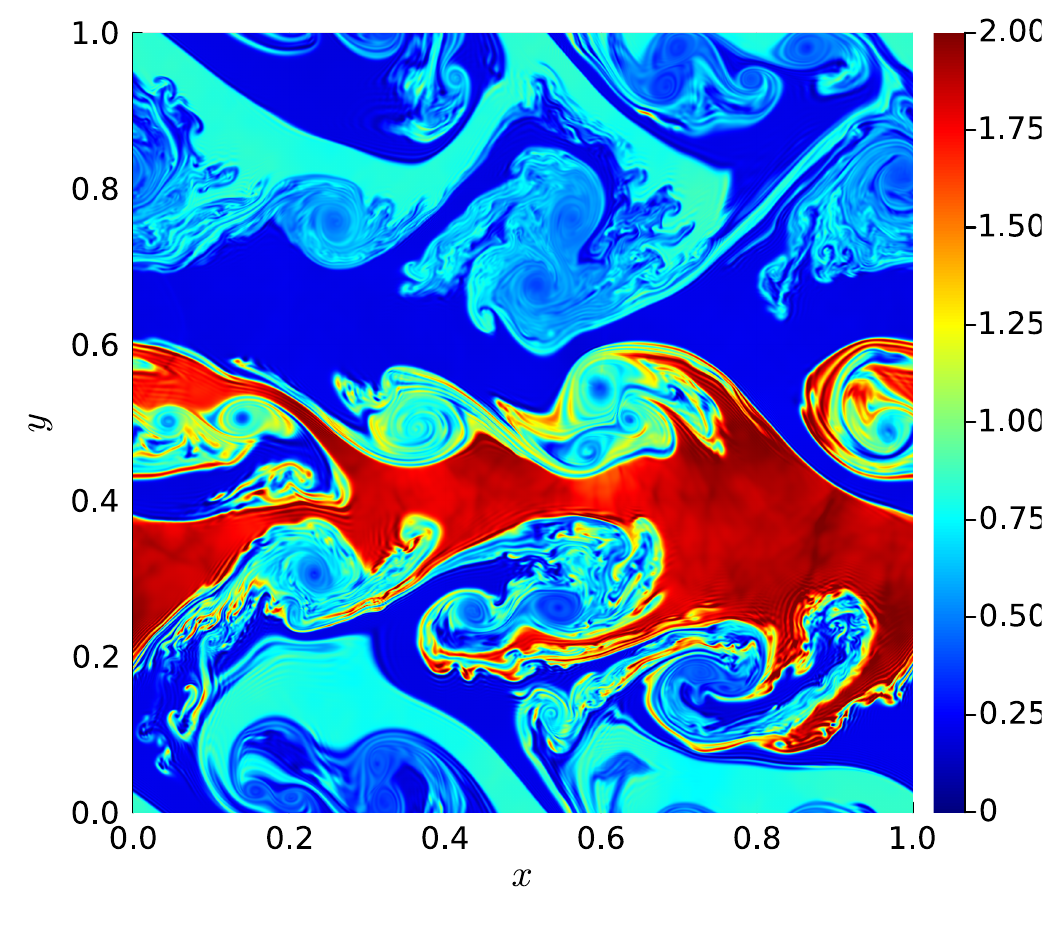}
\includegraphics[width=0.32\textwidth]{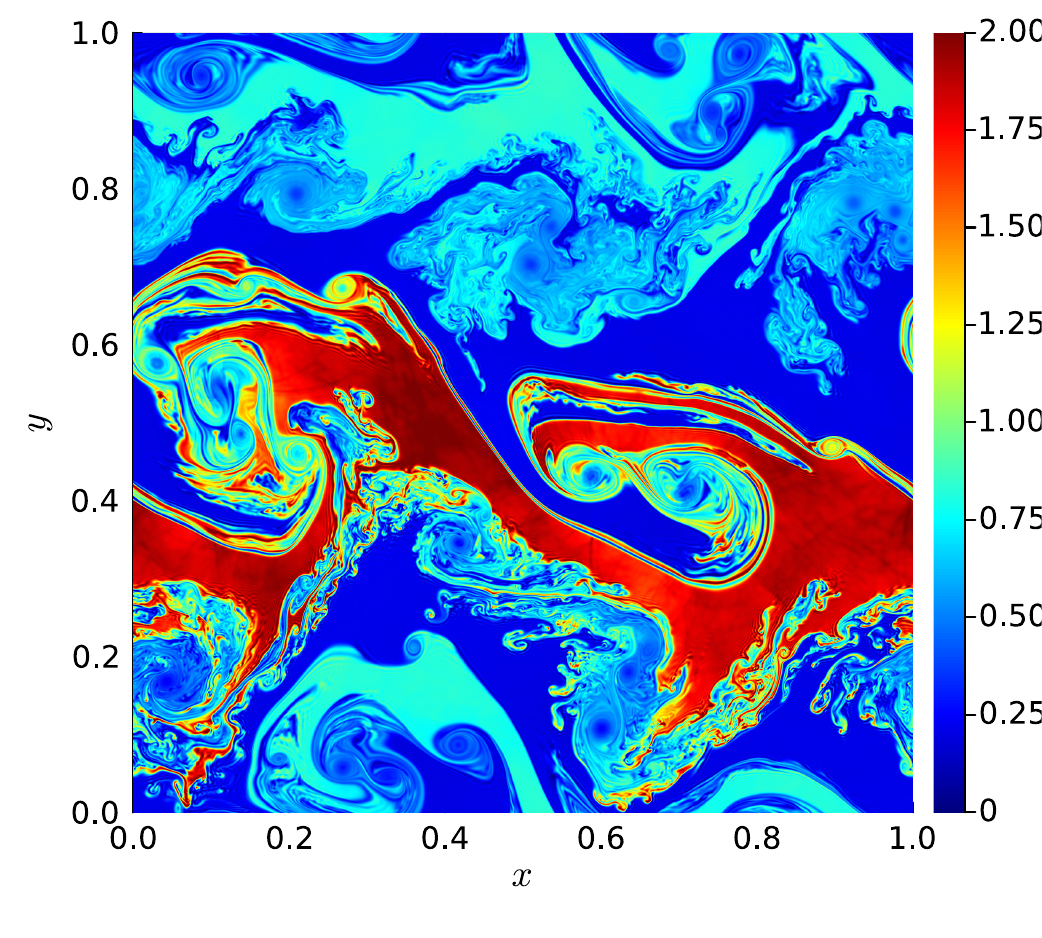}\\[4pt]
\includegraphics[width=0.32\textwidth]{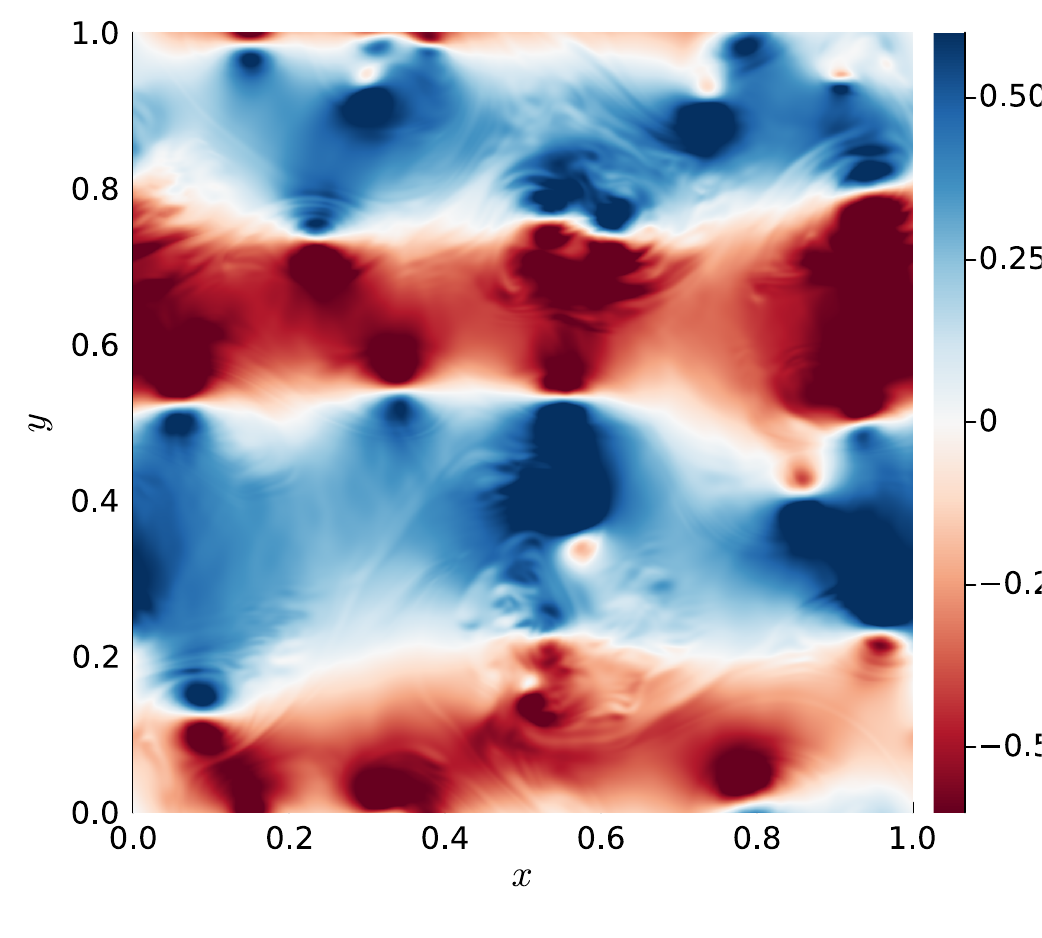}
\includegraphics[width=0.32\textwidth]{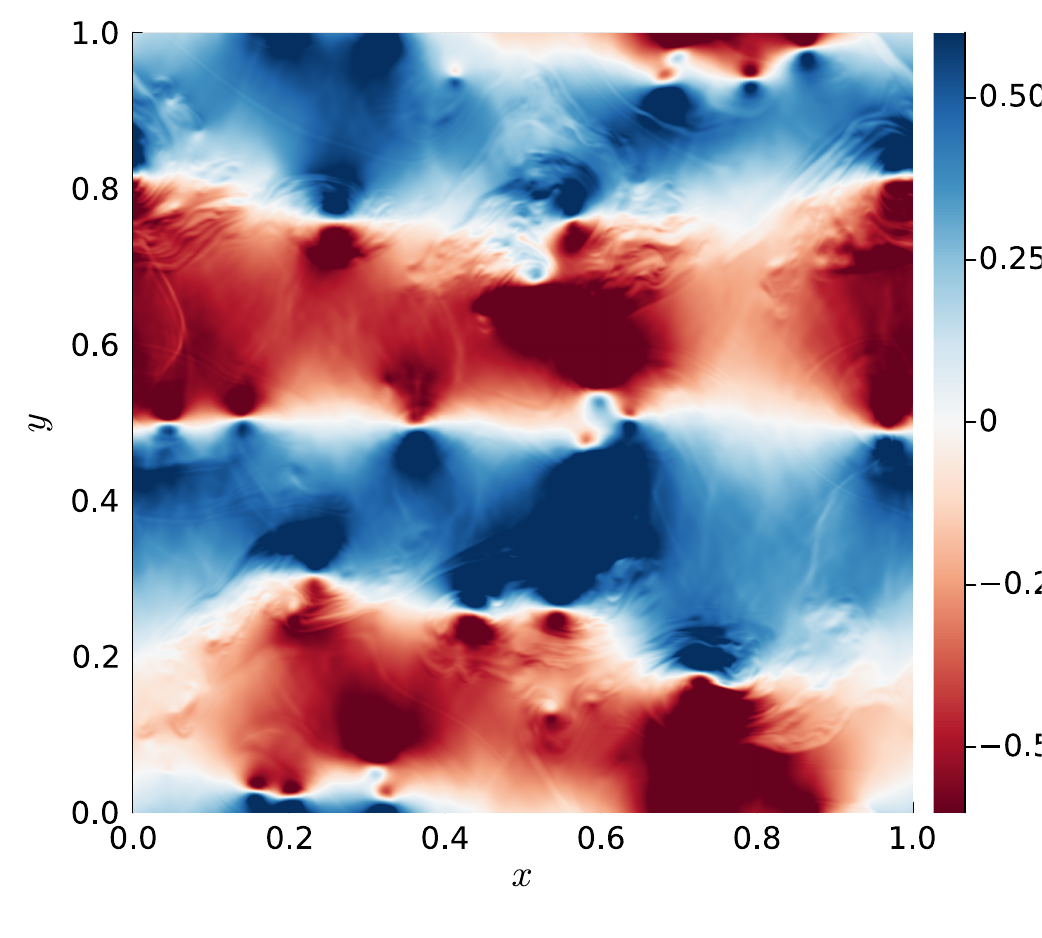}
\includegraphics[width=0.32\textwidth]{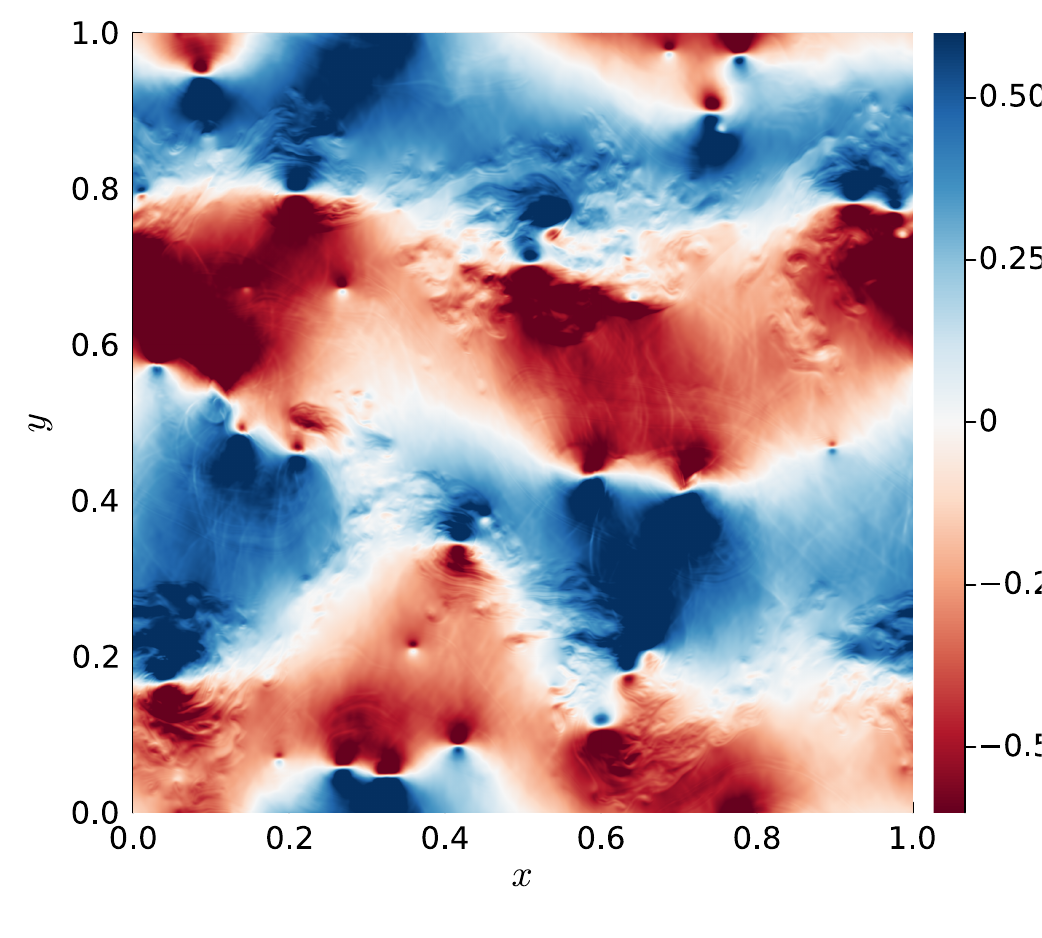}
\caption{KHI test at time $t = 2$, $polydeg = 1$. 
Pseudocolor plots of $\rho_{1,h}$ (top row), 
$\rho_{2,h}$ (middle row), and $\vel_{1,h}$ (bottom row) 
on meshes $512^2$, $1024^2$, $2048^2$ (left to right). 
The fine-scale mixing structures differ visually at 
each mesh level, consistent with the non-convergence 
of $E_1$ reported in Table~\ref{tab:khi}.}
\label{fig:khi_p1}
\end{figure}

\subsection{Shock-Bubble Interaction}

Finally, we validate the scheme on the two-dimensional shock-bubble interaction for completeness~(see \cite{gouasmi2020formulation} for details of the physical setup and we use the implementation available in \texttt{Trixi.jl} \cite{schlottkelakemper2025trixi}.)
The computational domain $\Omega = [-2.25,\,2.20] \times 
[-2.225,\,2.225]$ contains two species: dry air 
($\gamma_1 = 1.4$, $r_1 = 0.287$ kJ/(kg$\cdot$K)) and 
a helium-air mixture ($\gamma_2 = 1.648$, 
$r_2 = 1.578$ kJ/(kg$\cdot$K)). A Mach~$1.22$ planar 
shock, initially located at $x_1 = 0.5$, propagates 
leftward and interacts with a helium bubble of 
radius $R = 0.25$ centered at the origin. The initial 
states in the three regions are taken from~\cite{gouasmi2020formulation} with positivity 
parameter $\delta\rho = 0.03$.

The spatial discretisation uses the nodal DGSEM with 
polynomial degree $polydeg = 2$ on a uniform mesh of 
$512 \times 512$ with periodic 
boundary conditions, implemented within the \texttt{Trixi.jl} framework. The scheme employs flux 
differencing with the volume 
flux~\cite{ranocha2018comparison}, the local 
Lax-Friedrichs surface flux, and the subcell 
shock-capturing of~\cite{hennemann2021provably} with 
$\alpha_{\max} = 0.5$, $\alpha_{\min} = 0.001$. 
Time integration uses the five-stage fourth-order 
low-storage Runge-Kutta scheme~\cite{kennedy2000low} at $\mathrm{CFL} = 0.3$ up to $t = 0.01$.

Figure~\ref{fig:shock_bubble} shows the solution at $t = 0.01$. 
The partial densities $\rho_1$ and $\rho_2$ capture the bubble deformation and Richtmyer-Meshkov instability development and produces results in significant agreement with \cite{gouasmi2020formulation, renac2021entropy}.

\begin{figure}[htbp]
\centering
\includegraphics[width=0.47\textwidth]{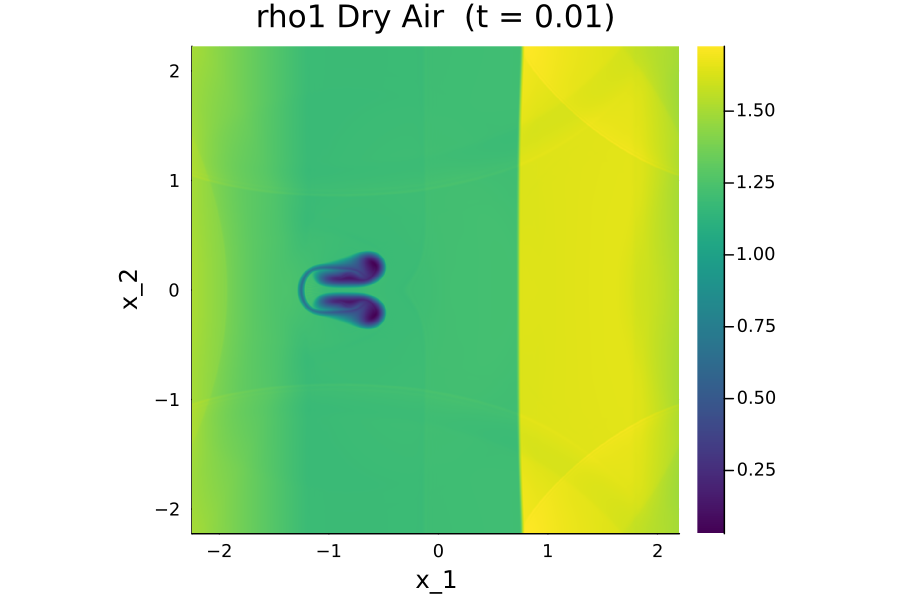}
\includegraphics[width=0.47\textwidth]{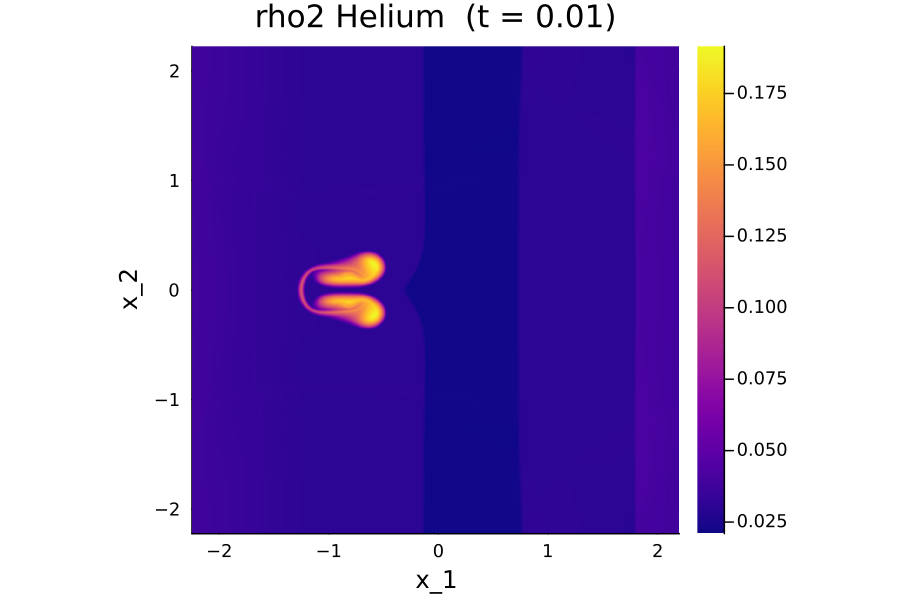}
\caption{Shock-bubble interaction at $t = 0.01$: partial 
densities $\rho_1$  and  $\rho_2$.}
\label{fig:shock_bubble}
\end{figure}
\section{Conclusion}

We have presented a convergence analysis of entropy-stable 
FV schemes for the multicomponent 
compressible Euler system in the framework of DW solutions. The scheme preserves the key physical 
structure of the mixture: positivity of all partial densities 
(Theorem~\ref{thm:density_positivity}), positivity of 
pressure and temperature (Lemma~\ref{pos_pressure}), 
and uniform stability estimates for the total density, 
momentum, and energy under some physical assumptions. These 
properties, combined with the weak BV estimate 
\eqref{BV_bound}, provide the arguments to prove convergence.
Furthermore, by applying the relative entropy framework, we demonstrate as well  the weak-strong 
uniqueness principle (Theorem~\ref{thm:WSU}) for completeness, and demonstrate that the numerical solutions converge strongly as long as strong solutions exist.  Numerical experiments for two components 
confirm the expected order of convergence  for smooth solutions and 
convergence of Ces\`aro averages for the 
Kelvin-Helmholtz instability, consistent with the 
theoretical predictions and similar to convergence results of structure-preserving numerical schemes for the Euler equations. 
A key requirement in our analysis was that no vacuum appears in any partial density. For future research,  it will be interesting if one relaxes this property by investigating the case where only the total density has to be bounded away from zero and partial densities can vanish. This is also related to designing adequate numerical schemes in such a context. 

\appendix
\section{Entropy Dissipation}\label{app:entropy_dissipation}

In this appendix, we justify the dissipation mechanism in the semi-discrete FV scheme in a general setting. Recalling the definitions~\eqref{ent_prod2}, and inserting them into the left-hand side of~\eqref{disc_entrorpy_inq}, the entropy production in a cell $K$ may be written as
\begin{equation}\label{eq:entropy_prod_decomp}
\mathcal{P}_K = \sum_{L \in \mathcal{N}(K)} \frac{|S_{KL}|}{2}
\bigl(\entpot_K - \entpot_L\bigr)
- \sum_{L \in \mathcal{N}(K)} \frac{|S_{KL}|}{2}\lambda_{KL}
\Bigl[\entvar_K^T(\con_L - \con_K) - (\ent_L - \ent_K)\Bigr],
\end{equation}
where $\entpot = \entvar^T\conflux - \entf$ is the entropy potential. The conservative contribution is antisymmetric across neighbouring cells and therefore contributes only to the numerical entropy flux. Entropy stability thus depends on the sign of the dissipative term.

By Taylor's theorem applied to $\ent$, there exists 
$\tilde{\con}$ on the segment joining $\con_K$ and $\con_L$ 
such that
\begin{equation}\label{eq:taylor_entropy}
\ent(\con_L) = \ent(\con_K) + \nabla_{\con}\ent(\con_K)^T(\con_L - \con_K)
+ \frac{1}{2}(\con_L - \con_K)^T \mathbf{H}(\tilde{\con})(\con_L - \con_K),
\end{equation}
where $\mathbf{H}(\con) = \nabla^2_{\con}\ent(\con)$ is the 
entropy Hessian. Since $\nabla_{\con}\ent = \entvar$, 
rearranging \eqref{eq:taylor_entropy} gives
\begin{equation}\label{eq:taylor_rearranged}
\entvar_K^T(\con_L - \con_K) - (\ent_L - \ent_K)
= -\frac{1}{2}(\con_L - \con_K)^T \mathbf{H}(\tilde{\con})(\con_L - \con_K).
\end{equation}

By Assumption~\ref{as_1}, Theorem~\ref{thm:density_positivity}, 
and Lemma~\ref{pos_pressure}, the solution satisfies 
$\rho_{i,h} \geq \underline{\varrho} > 0$, $\temp_h \geq T_* > 0$, 
and $\pressure_h \geq p_* > 0$ uniformly. For such physically 
admissible states, the entropy $\ent = -\rho s$ is strictly 
convex in conservative variables, so $\mathbf{H}(\tilde{\con})$ 
is positive definite:
\begin{equation}\label{eq:hessian_pd}
(\con_L - \con_K)^T \mathbf{H}(\tilde{\con})(\con_L - \con_K) \geq 0.
\end{equation}

Therefore, the dissipative contribution is non-negative:
\begin{equation}\label{r_KL_def}
-\frac{\lambda_{KL}}{2}\bigl[\entvar_K^T(\con_L - \con_K) - (\ent_L - \ent_K)\bigr]
= \frac{\lambda_{KL}}{4}(\con_L - \con_K)^T \mathbf{H}(\tilde{\con})(\con_L - \con_K)
=: r_{KL} \geq 0.
\end{equation}

Collecting the above observations, we infer that the semi-discrete scheme satisfies the cellwise entropy inequality
\begin{equation}\label{eq:discrete_entropy_ineq_final}
\frac{d}{dt}\ent(\con_K) + \frac{1}{|K|}\sum_{L \in \mathcal{N}(K)}
|S_{KL}|\hat{\entf}_{KL} \leq 0,
\end{equation}
where $\hat{\entf}_{KL}$ is the numerical entropy flux 
defined by~\eqref{ent_prod2}. This establishes entropy 
stability of the scheme.


\section{Weak BV Estimate}\label{app:weakBV}

\begin{proposition}[Weak BV estimate]\label{prop:weak_BV}
Under Assumption~\ref{as_1}  the semi-discrete Lax-Friedrichs scheme \eqref{eq:semi_scheme}-\eqref{eq:LF_flux_def} with viscosity parameter satisfying~\eqref{local_diff} satisfies the weak BV estimate \eqref{BV_bound}.
\end{proposition}
\begin{proof}
By Assumption~\ref{as_1}, Theorem~\ref{thm:density_positivity}, 
and Lemma~\ref{pos_pressure}, the entropy $\ent = -\rho s$ 
satisfies
\begin{equation}\label{eq:unif_conv}
D^2_{\con}\ent(\con_h) \geq \alpha\,\mathbb{I}
\quad \text{a.e., for some } \alpha > 0 \text{ independent of } h.
\end{equation}
The Lax-Friedrichs flux is entropy 
stable, yielding the discrete 
entropy inequality
\begin{equation}\label{eq:discrete_entropy_ineq}
\frac{d}{dt}\sum_K\ent(\con_K)|K| + 
\sum_{KL\in\mathcal{E}_h}\hat{\entf}_{KL}|S_{KL}| 
\leq -\sum_{KL\in\mathcal{E}_h}r_{KL}|S_{KL}|,
\end{equation}
where $r_{KL} \geq 0$ is the entropy production~\eqref{r_KL_def}. 
Integrating over $(0,T)$ gives 
$$\int_0^T\sum_{KL}|S_{KL}|\,r_{KL}\,dt \leq C,$$
By the renormalized entropy argument~\eqref{eq:renorm-entropy}(see also~\cite{feireisl2021numerical} for details),
\begin{equation}\label{eq:entropy_prod_vanish}
\int_0^T\sum_{KL\in\mathcal{E}_h} h^N r_{KL}(t)\,dt 
\longrightarrow 0 \quad\text{as }h\to 0.
\end{equation}
The mean value theorem and~\eqref{eq:unif_conv} yield
\begin{equation}\label{eq:r_lower_bound}
r_{KL} \geq \frac{\alpha}{4}\lambda_{KL}|[\con]_{KL}|^2,
\quad [\con]_{KL} := \con_L - \con_K.
\end{equation}
Combining \eqref{eq:entropy_prod_vanish} and 
\eqref{eq:r_lower_bound},
\[
\int_0^T\sum_{KL}h^N\lambda_{KL}|[\con]_{KL}|^2\,dt 
\longrightarrow 0.
\]
By the Cauchy--Schwarz inequality, we obtain
\[
\int_0^T\sum_{KL}h^N\lambda_{KL}|[\con]_{KL}|\,dt
\leq \Bigl(\int_0^T\sum_{KL}h^N\lambda_{KL}\,dt\Bigr)^{1/2}
\Bigl(\int_0^T\sum_{KL}h^N\lambda_{KL}|[\con]_{KL}|^2\,dt\Bigr)^{1/2}.
\]
The second factor vanishes. The first is bounded by the 
discrete trace inequality~\cite{feireisl2020convergence} 
and Assumption~\ref{as_1}: 
$\int_0^T\sum_{KL}h^N\lambda_{KL}\,dt \leq C$, 
yielding~\eqref{BV_bound}.
\end{proof}

\subsection*{Acknowledgment }
The work of Jaya Agnihotri and Philipp Öﬀner was supported by the German Research Foundation (DFG) under the   grant OE 661/4-1(520756621).

\bibliographystyle{plain}
\bibliography{reference}	

\end{document}